\documentclass[11pt,english]{article}

\usepackage[margin = 2.5cm]{geometry}

\usepackage{amsmath}
\usepackage{amsthm}
\usepackage{amssymb}
\usepackage{mathtools}
\usepackage{graphicx}
\usepackage{xcolor}
\usepackage[hidelinks, bookmarks = false]{hyperref}
\usepackage{enumitem}
\usepackage{bbm}

\usepackage{caption}
\usepackage[square,sort,comma,numbers]{natbib} 
\usepackage{tikz} 

\usepackage{setspace}
\usepackage{cleveref}
\theoremstyle{plain}

\newtheorem{theorem}{Theorem}[section]
\crefname{theorem}{Theorem}{Theorems}
\Crefname{theorem}{Theorem}{Theorems}

\newtheorem*{lemma*}{Lemma}
\newtheorem{lemma}[theorem]{Lemma}
\crefname{lemma}{Lemma}{Lemmas}
\Crefname{lemma}{Lemma}{Lemmas}

\newtheorem*{claim*}{Claim}
\newtheorem{claim}[theorem]{Claim}
\crefname{claim}{Claim}{Claims}
\Crefname{claim}{Claim}{Claims}

\crefname{proposition}{Proposition}{Propositions}
\Crefname{proposition}{Proposition}{Propositions}

\newtheorem{corollary}[theorem]{Corollary}
\crefname{corollary}{Corollary}{Corollaries}
\Crefname{corollary}{Corollary}{Corollaries}

\newtheorem{conjecture}[theorem]{Conjecture}
\crefname{conjecture}{Conjecture}{Conjectures}
\Crefname{conjecture}{Conjecture}{Conjectures}

\crefname{question}{Question}{Questions}
\Crefname{question}{Question}{Questions}

\crefname{observation}{Observation}{Observations}
\Crefname{observation}{Observation}{Observations}

\crefname{example}{Example}{Examples}
\Crefname{example}{Example}{Examples}

\theoremstyle{definition}

\crefname{problem}{Problem}{Problems}
\Crefname{problem}{Problem}{Problems}

\newtheorem{definition}[theorem]{Definition}
\crefname{definition}{Definition}{Definitions}
\Crefname{definition}{Definition}{Definitions}

\usepackage{xpatch}
\xpatchcmd{\proof}{\itshape}{\normalfont\proofnamefont}{}{}
\newcommand{\proofnamefont}{}
\renewcommand{\proofnamefont}{\bfseries}

\usepackage{framed}

\newcommand{\remove}[1]{}

\newcommand{\floor}[1]{
    \lfloor #1 \rfloor
}

\newcommand{\del}{\delta}
\newcommand{\cL}{\mathcal{L}}
\newcommand{\M}{\mathcal{M}}
\newcommand{\F}{\mathcal{F}}
\newcommand{\T}{\mathcal{T}}

\DeclareMathOperator{\prob}{Prob}

\newcommand{\HH}{\mathcal{H}}
\newcommand{\A}{\mathcal{A}}
\newcommand{\B}{\mathcal{B}}
\newcommand{\C}{\mathcal{C}}
\newcommand{\D}{\mathcal{D}}
\newcommand{\G}{\mathcal{G}}
\newcommand{\K}{\mathcal{K}}
\newcommand{\Z}{\mathcal{Z}}
\newcommand{\I}{\mathcal{I}}
\newcommand{\NN}{\mathbb{N}}
\newcommand{\EE}{\mathbb{E}}

\newcommand{\rank}{\mathrm{r}}

\newcommand{\Pa}{\mathcal{P}}
\newcommand{\Sa}{\mathcal{S}}

\newcommand{\ve}{\varepsilon}

\newcommand{\wB}{\widetilde{\mathcal{B}}}
\newcommand{\W}{\mathcal{W}}

\newcommand{\sub}{\subseteq}

\providecommand{\keywords}[1]
{
  \small	
  \textbf{\textit{Keywords: }} #1
}

\providecommand{\subjectclass}[1]
{
  \small	
  \textbf{\textit{2020 MSC subject classification: }} #1
}

\title{Tree Posets: Supersaturation, Enumeration, and Randomness }
\author{Tao Jiang \footnote{Dept. of Mathematics, Miami University, Oxford, OH 45056, USA, {\tt jiangt@miamioh.edu}. Research supported
by the National Science Foundation grant DMS-1855542.  }\and Sean Longbrake \footnote{Dept. of Mathematics, Emory University,  Atlanta, GA 30322, USA {\tt sean.longbrake@emory.edu}} \and Sam Spiro\footnote{Dept.\ of Mathematics, Georgia State University {\tt sspiro@gsu.edu}. This material is based upon work supported by the National Science Foundation Mathematical Sciences Postdoctoral Research Fellowship under Grant No. DMS-2202730.} \and Liana Yepremyan\footnote{Dept.\ of Mathematics, Emory University {\tt lyeprem@emory.edu}.  Research is supported by the National Science Foundation grant 2247013: Forbidden and Colored Subgraphs.}}

\begin{document}

\maketitle
\begin{abstract}
We develop a powerful tool for embedding any tree poset $P$  of height $k$ in the Boolean lattice which allows us to solve several open problems in the area.  We show that:

\begin{itemize}
    \item If $\mathcal{F}$ is a family in $\mathcal{B}_n$ with $|\mathcal{F}|\ge (q-1+\varepsilon){n\choose \lfloor n/2\rfloor}$ for some $q\ge k$, then $\mathcal{F}$ contains on the order of as many induced copies of $P$ as is contained in the $q$ middle layers of the Boolean lattice. This generalizes results of Bukh \cite{Bukh} and Boehnlein and Jiang \cite{BJ} which guaranteed a single such copy in non-induced and induced settings respectively.
    \item The number of induced $P$-free families of $\mathcal{B}_n$ is $2^{(k-1+o(1)){n\choose \lfloor n/2\rfloor}}$, strengthening recent independent work of Balogh, Garcia, Wigal \cite{BGW} who obtained the same bounds in the non-induced setting.
    \item The largest induced $P$-free subset of a $p$-random subset of $\mathcal{B}_n$ for $p\gg n^{-1}$ has size at most $(k-1+o(1))p{n\choose \lfloor n/2\rfloor}$, generalizing previous work of Balogh, Mycroft, and Treglown \cite{BMT} and of Collares and Morris \cite{CM} for the case when $P$ is a chain.
\end{itemize}

All three results are asymptotically tight and give affirmative answers to general conjectures of Gerbner, Nagy, Patk\'os, and Vizer \cite{GNPV} in the case of tree posets.

\end{abstract}


\subjectclass{05D05.}

\section{Introduction}

The celebrated Sperner's theorem \cite{Sperner} in extremal set theory determines the size of the largest family of sets in $[n]$ not containing
a $2$-chain $F_1\supset F_2$. Later, Erd\H{o}s \cite{Erdos}
 extended Sperner's theorem to determine the largest family not containing any $k$-chain and showed how Sperner's lemma can be used to solve the  classical Littlewood–Offord problem~\cite{LO}.   Afterwards,  Katona and Tarj\'an \cite{KT} initiated a systematic study of  the size of the largest family in the Boolean lattice $\B_n$ that avoids a given subposet.  This topic has attracted much attention and witnessed many advances in the last decades; we refer the interested reader to the nice survey paper by Griggs and Li \cite{GLi} for more.

 The question we are interested in this paper is one of \emph{supersaturation}, that is, how many copies of a poset are we guaranteed once we are above the threshold of containing one. The simplest poset to consider the supersaturation question for are $2$-chains. Erd\H{o}s and Katona conjectured that a family with $\binom{n}{\lfloor n/2 \rfloor}+t$ sets in $[n]$ must contain at least $t\cdot \lfloor \frac{n+1}{2}\rfloor$ many $2$-chains. This conjecture was confirmed by Kleitman~\cite{Kleitman}, who in fact showed that for every $0\leq a\leq 2^n$, every family in $[n]$ with size $a$ contains at least as many $2$-chains as the so-called centralized family of size $a$, i.e. a family of $a$ sets whose cardinalities are as close to $n/2+1/4$ as possible. Kleitman~\cite{Kleitman} further conjectured that the same families should also minimize the number of $k$-chains for every $k$. Five decades later, Kleitman's result was rediscovered by Das, Gan, and Sudakov \cite{DGS} and independently by Dove, Griggs, Kang, and Sereni \cite{DGKS}. Both papers further confirmed Kleitman's conjecture for every $k$ and $a$ belonging to a certain range above the sum
of the $k-1$ largest binomial coefficients. Subsequently, Balogh and Wagner \cite{BW} proved Kleitman's conjecture for all $k$ and
$a\leq (1-\ve)2^n$, provided that $n$ is sufficiently large with respect to $k$ and $\ve$. Finally, in a remarkable paper,  Samotij~\cite{Samotij} resolved Kleitman's conjecture in full. 
There have been further generalizations of  the supersaturation problem for $2$-chains to more general hosts other than the Boolean lattice such as the collection of subspaces of $\mathbb{F}_q^n$ ordered by set inclusion \cite{NSS} or $\{0,1,2, \dots r\}^n$ \cite{BW,NSS}.  The latter problem gives rise to a natural generalization of Kleitman's problem,  and while weaker approximate results do hold for  $r\geq 2$ as shown by Noel, Sudakov and Scott~\cite{NSS},  the exact analogue of Kleitman's conjecture (see ~\cite{BW,NSS}) fails as shown by Balogh, Petříčková, and Wagner~\cite{BPW}.

Our main goal is to establish similar supersaturation results for more general family of  posets beyond 2-chains, and in particular we do this for so-called \emph{tree posets}, which are posets whose Hasse diagram is a tree. To state our results, we need a short prelude on the precise definition of extremal numbers for posets, since below this extremal number we can not guarantee any copies of our poset.   

 Let P, Q be two finite posets, that is, they are finite sets equipped with partial orders $<_{P}$ and $<_Q$. A {\it poset homomorphism} is a function $f: P\to Q$ such that $f(x) <_Q f(y)$
whenever $x <_P y$. An {\it induced poset homomorphism} is a function $f: P\to Q$ such that $f(x) <_Q f(y)$ if and only if $x <_P y$. We say that a poset $Q$ {\it contains} another poset $P$ if there is an injective poset homomorphism from $P$ to $Q$. We say that a poset $Q$ contains an induced copy of another poset $P$ if there is an injective induced poset homomorphism from $P$ to $Q$.  If a poset $Q$ does not contain a copy of another poset $P$, we say that $Q$ is {\it $P$-free}. If $Q$ does not contain an induced copy of $P$, we say that $Q$ is  {\it induced $P$-free}. Given a poset $P$ and an integer $n$, we define $La(n,P)$ to be the largest size of a $P$-free subfamily of $\B_n$  and $La^*(n,P)$ the largest size of an induced $P$-free subfamily of $\B_n$. Motivated by a number of early results in the study of $La(n,P)$ (see for example \cite{DK, DKS, GK, Th}), Griggs and Lu \cite{GL} and independently Bukh \cite{Bukh}  made the following conjecture on the form of $La(n,P)$. 

\begin{conjecture} [Bukh\cite{Bukh}, Griggs-Lu \cite{GL}] \label{poset-turan-conjecture}
Let $P$ be a poset. Then
\[La(n,P)=(1+o(1))e(P)\binom{n}{\lfloor n/2 \rfloor},\]
where $e(P)$ denotes the largest integer $\ell$ such that for all $j$ and $n$ the family $\bigcup_{i=1}^\ell \binom{[n]}{i+j}$
is $P$-free.
\end{conjecture}

A similar conjecture for the induced case as well as a supersaturation version of both results were stated in Gerbner, Nagy, Patk\'os, and Vizer (see Conjecture~1 and Conjecture~3 ~\cite{GNPV}). An approximate version of Conjecture~\ref{poset-turan-conjecture} was proven by  Methuku and P\'alv\"olgyi \cite{MP} who showed that for every poset $P$, there exists a constant $C_P$ such that $La^*(n,P)\leq C_P\binom{n}{\lfloor n/2\rfloor}$. The value of $C_P$  was later improved by Tomon~\cite{Tomon}. However, this conjecture as well as the companion ones from~\cite{GNPV} were proven to be false for $P$ the $d$-dimensional Boolean lattice with $d\geq 4$ due to a recent beautiful construction of Ellis, Ivan, and Leader \cite{EIL}.

Nevertheless there are still a number of families of $P$ for which \Cref{poset-turan-conjecture} remains true (see \cite{GLi} for more), the most general family perhaps being that of tree posets as established by Bukh~\cite{Bukh}.

\begin{theorem}[Bukh \cite{Bukh}]
Let $P$ be a tree poset of height $k$. Then
\[La(n,P)=\left(k-1+O\left(\frac{1}{n}\right)\right)\binom{n}{\lfloor n/2\rfloor}.\]
\end{theorem}
Bukh's result was later extended to the induced setting by Boehnlein and Jiang \cite{BJ}, who
showed that $La^*(n,P)=(k-1+O(\frac{\sqrt{n\log n}}{n}))\binom{n}{\lfloor n/2\rfloor}$ for any tree poset of height $k$. Our main result is the following supersaturation extension of these
results.

\begin{theorem}\label{thm:main}
Let $k$ be a fixed positive integer and $P$ a tree poset of height $k$.
Then for any real $\ve>0$ and integer $q\geq k$, there exists a real $\delta>0$
such that every family
$\F\subseteq \B_n$ with $|\F|\geq (q-1+\ve)\binom{n}{\lfloor n/2\rfloor}$ contains at least $\delta\cdot M^*(n,q,P)$
induced copies of $P$ where $M^*(n,q,P)$ denotes the number of induced copies of $P$ in the $q$ middle levels of $\B_n$.
\end{theorem}

This result answers \cite[Conjecture 3]{GNPV} in a strong form for tree posets which stated their conjecture only in the special case of $q=k$.  In fact, an analog of \Cref{thm:main} holds if we replace $|\mathcal{F}|$  with $\mathcal{F}$'s Lubell weight, see Theorem~\ref{prop:embedding} for the precise statement. It is worth noting that we do not know of an explicit formula for $M^*(n,q,P)$ in general, but this does not end up being a significant barrier to the proof (see \Cref{sec:sketch} for more details). However, as a corollary to \Cref{thm:main}, we can get the following explicit result which is tight for \emph{saturated tree posets}, i.e. those for which every maximal chain has the same length.
\begin{corollary}\label{cor:supersat}
Let $k$ be a fixed positive integer and $P$ a tree poset of height $k$.
For any real $\ve>0$, there exists a real $\delta>0$ such that every subfamily
$\F\subseteq \B_n$ with $|\F|\geq (k-1+\ve)\binom{n}{\lfloor n/2 \rfloor}$ contains at least 
$\delta n^{|P|-1}\binom{n}{\lfloor n/2 \rfloor}$ induced copies of $P$.
\end{corollary}

A slight modification of our approach can extend \Cref{cor:supersat} into a \textit{balanced supersaturation} result, \Cref{thm:balanced}, which roughly speaking says that we can guarantee our collection of copies of $P$ to be such that  no subset of $\F$ is contained in too many induced copies of $P$ in this collection. Balanced supersaturation  while interesting on its own is usually used in combination with the  container method to establish counting results. The container method originated in papers by Kleitman and Winston~\cite{KW1, KW2} in early 1980s and was further  independently developed by Balogh, Morris, and Samotij~\cite{BMS} and Saxton and Thomason~\cite{ST}, and has had a tremendous impact on combinatorics since then.  Containers have been widely used to establish counting results and random Tur\'an-type  in various settings  such as for $H$-free graphs or hypergraphs for a fixed graph $H$, AP-free sets in additive combinatorics and so on. See the excellent survey by Balogh, Morris and Samotij~\cite{survey-container} for an extensive overview on the method. 

In the setting of posets, containers have been most widely used in counting subsets of $\B_n$ avoiding $2$-chains, which is equivalent to the problem of counting antichains. When the host is the Boolean lattice, this problem was  solved independently by Balogh, Treglown, and Wagner~\cite{BTW} and by Noel, Sudakov, and Scott~\cite{NSS}. The second group~\cite{NSS} also generalized this to hosts being subspaces of $\mathbb{F}_n^q$ ordered by inclusion, and sets of divisors of a square of a square free integer. Note that the latter is equivalent to studying $\{0, 1, 2\}^n$ under the natural ordering. In the more general setting of counting antichains in $[t]^n$, which turns out to be connected to a Ramsey-theoretic question in ordered hypergraphs~\cite{MS}, there has been much recent progress, most notably by Pohoata and Zakharov~\cite{PZ}, Park, Sarantis, and Tetali~\cite{PST} and finally, by Falgrav-Ravry, R\"avy, and Tomon~\cite{FRT}.

Our main counting result gives tight bounds for  the number of induced $P$-free families of $\B_n$ for all tree posets $P$, solving Conjecture 4 from \cite{GNPV} in this case.   This extends previous work of Patk\'os and Treglown~\cite{PT} and Gerbner, Nagy, Patk\'os, and Vizer \cite{GNPV} who obtained similar results for special subclasses of tree posets. Our result also strengthens the recent work of Balogh, Garcia, and Wigal~\cite{BGW} who independently obtained the same counting result but in the non-induced setting. 

\begin{theorem}\label{thm:counting}
    If $P$ is a tree poset of height $k$, then the number of induced $P$-free sets in $\B_n$ is at most 
    $$2^{(k - 1 + o(1))\binom{n}{n / 2}}.$$
\end{theorem}

With a similar approach, we also obtain tight bounds for the largest size of an induced $P$-free family of a random subset of $\B_n$ for all tree posets $P$.
Let $\Pa(n,  p)$ be the uniformly random subset of $\B_n$,  where each set survives with probability $p$ such that $pn \rightarrow \infty$. This model $\Pa(n,  p)$ was first investigated by Renyi~\cite{R}, after which Kohayakawa and Kreuter~\cite{KK} studied the size of the largest $2$-chain-free subset of $\Pa(n, p)$.  Their results were subsequently improved by Osthus~\cite{O} and Balogh, Mycroft and Treglown~\cite{BMT}, the latter establishing optimal bounds, and independently by Collares and Morris~\cite{CM} who established the analogous results for $k$-chains for all $k\geq 2$. Other results in this direction were obtained by  Patk\'os and Treglown~\cite{PT} and by Gerbner, Nagy, Patk\'os, and Vizer \cite{GNPV} for some special subfamilies of tree posets.  We establish a far-reaching generalization of all these  results by establishing the corresponding  result on $\Pa(n, p)$ for all tree posets $P$ in the induced setting, solving Conjecture 7 of \cite{GNPV} in this case.

\begin{theorem}\label{thm:randomturan}
 If $P$ is a tree poset and $pn \rightarrow \infty$, then with high probability, the largest induced $P$-free subset of $\Pa(n, p)$ has size $(k - 1 + o(1))p\binom{n}{n / 2}$. 
\end{theorem}
The hypothesis that $pn\to \infty$ is best possible for this result to hold.   Indeed, as noted in \cite{CM, BMT}, if $p=cn^{-1}$ then a computation due to Osthus~\cite{O} shows that with high probability there exists subsets of $\Pa(n,p)$ of size at least $(k-1+e^{-c}+o(1)){n\choose n/2}$ which contains no $k$-chain and thus also no tree poset of height $k$.

Finally, it is worth highlighting that the main tool of this paper, Theorem~\ref{thm:cleaning}, which we believe is of independent interest, and can be viewed as a general tool similar to results in graphs of passing to subgraphs of high minimum degree. While the exact analogue of such a result is out of reach, that is,  obtaining a subfamily in the poset setting of ``high minimum degree'' we obtain a sequence of nested subfamilies of our original family $\mathcal{F}$ having high minimum degree in the predecessor of the sequence which still allows us to greedily find many embeddings of a tree poset, similar to embedding copies of a fixed tree in a high minimum degree subgraph. We believe that this result will have further applications in the future.

The rest of this paper is organized as follows. In \Cref{sec:notation} we introduce some notation.  In Section~\ref{sec:sketch} we present a sketch of the proof of Theorem~\ref{thm:main}. In Section~\ref{sec:cleaning} after gathering some preliminary lemmas, we prove Theorem~\ref{thm:cleaning}, in Section~\ref{sec:embedding} we prove both the usual and balanced supersaturation results Theorem~\ref{thm:main} and Theorem~\ref{thm:balanced}. In Section~\ref{sec:containers} we use hypergraph containers together with balanced supersaturation to prove Theorem~\ref{thm:counting} and Theorem~\ref{thm:randomturan}.

\section{Notation}\label{sec:notation}

Throughout our paper, we drop floors and ceiling whenever these are not crucial to our analysis. Let $$\wB_n := \left\{ F \in \B_n : \left|F - \frac{n}{2}\right| < 2 \sqrt{n \ln n } \right\}. $$

Using some standard tools, we will assume that the family $\F$ we are working with is subfamily of $\wB_n$ instead of $\B_n$.

We define the Hasse diagram of $P$, denoted $H(P)$, as a directed graph with vertex set $P$ where there is an edge from $x$ to $y$ only if $y > x$ and there is no $z$ such that $y > z > x$. Notice that this definition differs slightly from the classical definition of Hasse diagram via undirected graphs embedded on the plane, but for us it will be more convenient to use the directed setting.

Given a subfamily $\F\subseteq \B_n$ and a positive integer $q$, we say that a tuple $(F_1,\ldots,F_q)$ of members of $\F$ is a {\it  $q$-chain} if $F_1\supset F_2\supset \cdots \supset F_q$. In particular, when we refer to the $i$th element of any $q$-chain we mean the one with $i$th largest cardinality. We will abuse notation slightly by occasionally identifying decreasing tuples $(F_1,\ldots,F_q)$ by the corresponding set $\{F_1,\ldots,F_q\}$.  We use $\C$ to denote the family of all full chains of $\B_n$ (with $n$ fixed), i.e.\ $\C$ is the set of chains of length $n+1$ in $\B_n$. The height of $P$ is the largest length of a chain in $P$.  

Given a $q$-chain $Q$ in $\F$ and a full chain $\chi$ in $\B_n$ that contains all of the members of $Q$, we 
call the pair $(\chi, Q)$ a {\it $q$-marked chain with markers in $\F$ } or {\it a $q$-marked chain from $\F$}. Given a family of $q$-marked chains $\M$, let 
\[\mathcal{L}^i(\M) = \{ D \in \mathcal{B}_n : D \text{ is  the } i\text{th member of a } q \text{-chain in } \M \}.\]

Our proof will rely on obtaining a nice nested sequence of $q$-marked
chains from $\F$ that are of a specific form.  To this end, given a family $\T$ of $1$-marked chains of some subfamily $\F\subseteq \B_n$ and $\chi\in \C$, we define $\T(\chi)=\{F: (\chi,F)\in \T\}$.
We say that $\T$ is {\it $q$-strong} if for each $\chi\in \C$ where $\T(\chi)\neq \emptyset$, we have that $|\T(\chi)|\geq q$.  
For a $q$-strong  $1$-marked chain family $\T$ from $\F$, we define the {\it $q$-th power} of $\T$, denoted by $\T[q]$, to be the $q$-marked chain family
 \begin{equation} \label{power-definition}
\T [q]= \left\{(\chi, Q): Q\in \binom{\T(\chi)}{q}\right\}
\end{equation}

We will consider some of our results in a somewhat more general setting.  To this end, we define the {\it Lubell weight} of a family $\mathcal{F}$, denoted by $\mu(\F)$, by
\begin{equation} \label{lubell-definition}
\mu(\F)=\sum_{F\in \F} \frac{1}{\binom{n}{|F|}}.
\end{equation}
The Lubell weight is a natural measure to put on families in $\B_n$, and in particular we note that $\mu(\F)$ is the expected number of members of $\F$ that are contained in a uniformly randomly chosen full chain. While this notion is not strictly needed for our work, we are able to apply our methodology in this more general setting which has historically been of interest to extremal problems for posets, see for example \cite{Meroueh}.

In Section~6, we will use the Chernoff's Inequality in the following form, see \cite{Chernoff}: 

\begin{lemma}\label{lem:chernoff}
 If $X = \mathrm{Bin}(n, p)$ and $\delta \leq 1$, 

 $$\prob(|X  - \EE[X])| \geq \delta  \EE[X] \leq 2 \exp\left( - \frac{\delta^2}{3} \EE[X]\right).$$
\end{lemma}

\section{Sketch of The Proof of Theorem~\ref{thm:main}}\label{sec:sketch}
Here we sketch how to prove \Cref{thm:main} which we recall says that every large family $\F \subset \B_n$ contains many copies of a given tree poset $P$. We will do this by showing that there are many ways of embedding $P$ into $\mathcal{F}$, but before we get into this we need some preliminaries.

It will be convenient to associate each member $F$ of a given subfamily $\F\subseteq \B_n$ with the full chains $\chi$ of $\B_n$ that contain it, so we will work with pairs $(\chi,F)$ instead of individual members $F$, where $F\in \mathcal{F}$ and $\chi$ is any full chain in the Boolean lattice. Likewise, we will associate each $q$-chain $Q$ in our family $\F$ with the full chains $\chi$ of $\B_n$ that contain it and call such pairs $(\chi, Q)$ $q$-marked chains.  It is worth mentioning that the  approach of working with $q$-marked chains was originated by Bukh~\cite{Bukh} and was later further developed by Boehnlein and Jiang~\cite{BJ}.  

Using Chernoff bounds, it is standard~\cite{BJ, GL} to show that the number of sets $F\in \B_n$ with $||F|-n/2|>2\sqrt{n\ln{n}}$ is $o({n \choose n/2})$. Thus, whenever we are given a dense subfamily $\F$ of $\B_n$ i.e. $|\mathcal{F}|=\Omega({n \choose n/2})$, by leaving out at most $o(\binom{n}{n/2})$ members of $\F$, we may assume that all $F\in \mathcal{F}$ lie in the family $\wB_n$.

At the heart of all of our arguments is a general result \Cref{thm:cleaning} which says that any subset $\F\sub \wB_n$ of large Lubell weight, more precisely at least of weight at least $q-1+\ve$,  contains a nested sequence of large $q$-marked  chain families  that have a certain robustness property which we will describe momentarily. Note that if $\mathcal{F}$ is of size at least $(q-1+\ve){n\choose \lfloor n/2\rfloor}$, then the hypothesis above  on the Lubell weight is satisfied. 

Before the embedding starts, we fix an ordering $x_1,x_2,\ldots,x_{|P|}$ of vertices of $P$ such that each $x_{j+1}$ has a unique neighbour (called the \emph{parent}) among $x_1, \dots x_{j}$, and we fix a poset homomorphism $\rank: P \rightarrow [q]$, with $[q]$ under the reverse of the natural total ordering. As a pre-processing step, we use our main tool \Cref{thm:cleaning} to generate a nested family of $q$-marked chains $\M^{0}\supseteq \M^{1}\supseteq \cdots \supseteq \M^{|P|}$ such that if $F\in \mathcal{F}$ is in some $q$-marked chain of $\mathcal{M}^i$ family then it is ``robust''  with respect to $\mathcal{M}^{i-1}$ as well, meaning $F$ is contained in ``many'' $q$-chains of $\mathcal{M}^{i-1}$ as well. The idea of nested families $\mathcal{M}^i$ may seem a bit peculiar but unfortunately we were not able to guarantee the existence of a single marked chain family $\mathcal{M}$ with the property of every $F$ being contained in the desired number of copies of $q$-marked chains in $\mathcal{M}$. To obtain this nested sequence, we iteratively remove all the sets  $F\in \mathcal{F}$ which are \emph{bad} with respect to $\mathcal{M}^{j-1}$ and the relevant $q$-chains from $\mathcal{M}^{j-1}$ subsequently, and  ensure that the sizes of $\mathcal{M}^j$ do not shrink dramatically, thus guaranteeing that after $|P|$ many steps we still have a large family of  $q$-marked chains.

Having run the cleaning process, we embed the vertices $x_1, x_2, \dots x_{|P|}$ iteratively so that at the $j$th step we find many partial embedding of the form  $\varphi^j: \{ v_1, v_2, \dots ,  v_j \}$. We embed $v_1$ arbitrarily in the  $\rank(v_1)$-level of any $q$-marked chain inside $\M^{|P|}$. Notice that the number of elements in the $\rank(v_1)$-level of a chain in $\M^{|P|}$ gives us a lower bound on the number of ways to embed $v_1$. Also while $\M^{|P|}$ is the smallest family of $q$-marked chains it has the property that every $F$ in a $q$-chain of $\mathcal{M}^{|P|}$ is ``robust" which respect to every $\mathcal{M}^j$ with $j <|P|$. This is  the incentive behind  embedding the vertices of $P$ in the reverse order of the nested families, more precisely, $v_j$ will be embedded in  $\mathcal{M}^{|P|-j+1}$.   At the $j$th step, we embed $v_{j}$, assuming that its parent $y$  who is among $\{v_1, v_2, \dots, v_{j-1}\}$ is already embedded.  We wish to embed $v_{j}$ in the $\rank(v_{j})$th position of some $q$-marked chain of $\mathcal{M}^{|P|-j+1}$ which has $\varphi^{j - 1}(y)$ in the $\rank(y)$th position.  By the cleaning process we ran earlier,  we are guaranteed to have many such choices to embed $v_j$ and find many partial embeddings $\varphi^{j}$.  This embedding procedure in total gives a lower bound on the number of induced copies of $P$ in $\F$ as a function of our choice of $\rank$. It turns out there is a relation between $M^*(n,q,P)$ and the counting  of copies of $P$ via all such rank functions $\rank$ (see Lemma~\ref{lem:rankandmiddle}). Thus we may choose $\rank$ suitably and get this lower bound to be a  constant fraction of $M^*(n,q,P)$.  This gives us our desired supersaturation result saying that the number of induced copies of $P$ guaranteed in $\F$ is as least as large as a fraction of the number of induced copies of $P$   in the middle $q$-levels without knowing $M^*(n,q,P)$ explicitly.

The full details of this embedding are slightly technical, as one also needs to make sure that the embedding guarantees noncomparable pairs stay noncomparable and other nuances, however the main gist of the argument is that the cleaning process in Theorem~\ref{thm:cleaning} provides the framework to do this successfully.

\section{Key Tools}\label{sec:cleaning}
In this section, we state and prove our main tool \Cref{thm:cleaning}, which roughly speaking guarantees in any large $\F\sub \B_n$ the existence of a nested sequence of $q$-marked chains $\M^{|P|}\sub \M^{|P|-1}\sub \cdots \sub \M^0$ which are ``robust'' in a certain way.  We begin by establishing some preliminary lemmas.

\subsection{Counting Lemmas on Marked Chains}\label{sub:counting}

In this section, we collect some basic counting lemmas about marked chains that we will use in the final step of our proof \Cref{prop:embedding}.  We begin by recalling the following lemma of  Bukh \cite[Lemma 4]{Bukh} for $q$-marked chains.

\begin{lemma}[\cite{Bukh}]\label{lem:Bukh}
    If $\F \subseteq \B_n$ and $|\F| \geq (q - 1 + \ve)\binom{n}{n / 2}$, then 
    there are at least $
    \frac{\ve}{q} n!$ $q$-marked chains with markers from $\F$. 
\end{lemma}

For our purposes we will need the following variant of \Cref{lem:Bukh}.

\begin{lemma} \label{size-to-weight}
Let $q$ be  a positive integer and $\ve>0$. Let $\F\subseteq \B_n$. Suppose $\mu(\F)\geq q-1+\ve$ and let
\[\T=\{(\chi,F): \chi\in \C, F\in \chi\cap \F, |\chi\cap \F|\geq q\}.\]
Then, $\T$ is a $q$-strong $1$-marked chain family from $\F$, satisfying $|\T|\geq \ve n!$.
In particular, if $|\F|\ge (q-1+\ve)\max_{F\in \F} {n\choose |F|}$ then $|\T|\geq \ve n!$
\end{lemma}
\begin{proof}
Let $\M=\{(\chi, F): \chi \in \C, F\in \F\}$. 
For each $i\in [n]$, let $C_i$ denote the number of full chains $\chi$ that contain exactly $i$ members of  $\F$.  
Then $|\M|=\sum_{i = 1}^n iC_i$. On the other hand, for each $F\in \F$, the number of full chains in $\B_n$ that contain $F$
is exactly $\frac{n!}{\binom{n}{|F|}}$. Hence,
 
\[\sum_{i = 1}^n iC_i=\sum_{F\in \F} \frac{n!}{{n\choose |F|}}=\mu(\F)n!\ge (q-1+\ve)n!.\]
Clearly, $\sum_{i<q} iC_i\leq (q-1)n!$.  Hence, $|\T|=\sum_{i\geq q} iC_i\geq \ve n!$.

 For the second statement, suppose $|\F|\geq (q-1+\ve) \max_{F\in \F}\binom{n}{|F|}$.
 Then 
 \[\mu(F)=\sum_{F\in \F}\frac{1}{\binom{n}{|F|}}\geq \frac{|\F|}{\max_{F\in \F} \binom{n}{|F|}} \geq q-1+\ve.\]
 So, the statement follows from the first statement.
\end{proof}
We note that although \Cref{size-to-weight} is stated in terms of a particular 1-marked chain family $\T$, it being $q$-strong immediately implies that $\F$ contains at least $\frac{\ve}{q} n!$ $q$-marked chains.  This recovers Bukh's original lemma that says if $|\F|\ge (q-1+\ve){n\choose n/2}$ then it contains at least $\frac{\ve}{q}n!$ $q$-marked chains and more generally recovers it under the weaker hypothesis $\mu(\F)\geq q-1+\ve$.

We will also need the following technical lemma, which roughly says that if $\T$ is a large $q$-strong 1-marked chain family, then there are many members of $\F$ that can start our embedding of $P$ at the ``$i$th level'' for all $i\in [q]$. 
 
 
\begin{lemma} \label{lem:translatingWeight}
Let $\F$ be a subfamily of $\B_n$ and $q$ a positive integer. Let $\T$ be a $q$-strong $1$-marked chain
  family from $\F$ and 
  let $\M=\T[q]$ be the $q$-th power of $\T$.  For each $i\in [q]$ and each $\chi\in \C$ let 
  \[\cL^i(\M,\chi):=\{F\in \F: \exists (\chi,Q)\in \M \text{ such that $F$ is the $i$-th member on $Q$}\},\]
  and let $\cL^i(\M)=\bigcup_{\chi\in \C} \cL^i(\M,\chi)$.
  If $|\T|\geq \ve n!$, then for each $i\in [q]$, we have \[|\cL^i(\M)|\geq \frac{\ve}{q} \min_{F\in \F} \binom{n}{|F|}.\]
\end{lemma}
\begin{proof}
   Consider any $\chi\in \C$ where $\T(\chi) \neq \emptyset$. Because $\T$ is $q$-strong, we have $|\T(\chi)|\geq q$ by definition. Note that
   for each $F\in \T(\chi)$ that is not among the largest $i-1$ members or smallest $q-i-1$ members in $\T(\chi)$,
   there exists a $q$-chain in $\binom{\T(\chi)}{q}$ that contains $F$ as the $i$-th member. Hence, $|\cL^i(\M,\chi)|\geq \frac{1}{q}|\T(\chi)|$
   and therefore 
   
   \begin{equation}\label{eqn:Li-upper}
   \sum_{\chi\in \C} |\cL^i(\M,\chi)|\geq \frac{1}{q} \sum_{\chi\in \C} |\T(\chi)|\geq \frac{\ve}{q} n!.
   \end{equation}
   On the other hand, for each $F\in \cL^i(\M)$, $F$ is contained in exactly $n!/\binom{n}{|F|}$ full chains of $\B_n$ and
   hence belongs to $\cL^i(\M,\chi)$ for  at most $n!/\binom{n}{|F|}$ different $\chi$. Hence,
   \begin{equation} \label{eqn:Li-lower}
   \sum_{\chi\in \C} |\cL^i(\M,\chi)|\leq \sum_{F\in\cL^i(\M)}\frac{n!}{\binom{n}{|F|}}\leq n!\cdot \frac{|\cL^i(\M)|}{\min_{F\in \F} \binom{n}{|F|}}. 
   \end{equation}
   Combining \eqref{eqn:Li-upper} and \eqref{eqn:Li-lower}, we get $|\cL^i(\M)|\geq \frac{\ve}{q} \min_{F\in \F} \binom{n}{|F|}$.
\end{proof}

\subsection{Main Cleaning Result}\label{sub:cleaning}

 Given a family $\F\sub \B_n$ and a family $\M$ of $q$-marked chains from $\F$, we define for each $i\in [q]$, $F\in \F$ and $\chi\in \C$ the sets

\[\M(\chi, F,i)=\{(\chi,Q)\in \M: F \text{ is the $i$-th member of } Q\},\]
\[\mathcal{M}(F,i)=\bigcup_{\chi \in \C}{\mathcal{M}(\chi, F,i)}.\]

\begin{definition}
For any $i\in [q]$, we say a member $F\in \B_n$ is {\it $(i, \delta)$-lower bad} with respect to $\mathcal{M}$  if $\M(F,i)\neq \emptyset$ and if there
exists a subfamily $\W\subseteq \B_n$ such that the following three properties hold:
\begin{enumerate}
    \item[$(a)$] Every $D \in \W$ satisfies $D \subseteq F$. 

    \item[$(b)$] For every $(\chi, Q) \in \M(F,i)$, we have $Q \cap \W \neq \emptyset$. 
    \item[$(c)$] We have
    $$\prob[\chi_0 \cap \W\neq \emptyset | F\in \chi_0 ]\leq \delta,$$
         where $\chi_0$ is a uniformly randomly chosen  full chain of $\B_n$.
\end{enumerate}
Any such subfamily $\W$ will be called an \emph{$(i,\delta)$-lower witness} for $F$.    
\end{definition}

Informally, $F$ being $(i,\delta)$-lower bad means that there exists a subfamily of small measure (in the sense of property (c)) $\W$ of members below $F$ such that every chain in $\M$ which has $F$ as the $i$th member must pass through this subfamily.  While there may be many subfamilies $\W$ which are $(i,\delta)$-lower witnesses for $F$, in some contexts it will be useful to work with some fixed canonical witness. To this end, for any $F$ that is $(i, \delta)$-lower bad with respect to $\M$, we let $\W(F,i, \M)$ denote the lexicographically minimal $\W$ which is an $(i,\delta)$-lower witness.
Note that for all $\W$, 
     \begin{equation} \label{eqn:lubell} 
     \prob[\chi_0 \cap {\W} \neq \emptyset | F \in \chi_0]\leq \sum_{D\in \W} \prob(D\in \chi_0| F\in \chi_0) =  \sum_{D \in {\W}} \frac{1}{\binom{ |F|}{|F - D|}},
     \end{equation} 
where the right hand side can be thought of as the Lubell weight of $\W$ relative to $F$. 

\begin{definition}
    Similarly, we say 
$F$ is {\it $(i, \delta)$-upper bad} with respect to $\mathcal{M}$  if there exists a set $\widehat{\W}$ such that the following three properties hold:  
\begin{enumerate}
    \item[($\hat{a}$)] Every $D \in \widehat{\W}$ satisfies $D \supseteq F$. 
       \item[$(\hat{b})$] For every $(\chi, Q) \in \M(F,i)$, we have $Q \cap \widehat{\W} \neq \emptyset$. 
    \item[$(\hat{c})$] We have $$\prob[\chi_0 \cap \widehat{\W}\neq \emptyset | F\in \chi_0 ] \leq \delta,$$
     where $\chi_0$ is a uniformly randomly chosen  full chain.
\end{enumerate}
Such $\widehat{\W}$ is called an  \emph{$(i, \delta)$-upper witness} for $F$.

\end{definition}

\begin{definition}\label{deltarobust}
    We say that a member $F\in \B_n$ is {\it $\delta$-robust} with respect to a $q$-marked chain family $\M$ if for
each $i\in [q]$, $F$ is neither $(i, \delta)$-lower-bad nor $(i, \delta)$-upper-bad with respect to $\M$.
\end{definition}

The rest of the section is dedicated to proving the following result which  builds a nested sequence of families of $q$-marked chains
with some robustness features with the additional property that each of these families is the $q$-th power of some
family of $1$-marked chains (recall the definition of the $q$-th power of a $1$-marked chain family from \eqref{power-definition}). This theorem provides the most important ingredient of our proof of Theorem~\ref{thm:cleaning}.

\begin{theorem} \label{thm:weakCleaning}
For all integers $q\ge 1$ and for all reals $\ve>0$, there exists some $\del>0$ such that the following holds.  Let $\F \subseteq \widetilde{\mathcal{B}}_n$ and let $\mathcal{T}^0$ be a $q$-strong $1$-marked chain family with markers from $\F$ such that $|\mathcal{T}^0| \geq \ve n!$. Then 
there exists a collection of $1$-marked chains $\mathcal{T}^{|P|} \subseteq \mathcal{T}^{|P| - 1} \subseteq \dots \subseteq \mathcal{T}^0 $ satisfying the following:
\begin{enumerate}
    \item For each $j=0,\dots, |P|$, $\T^j$ is $q$-strong.
    \item For each $j=1,\dots, |P|$, for each $(\chi,Q)\in \mathcal{T}^j[q]$ and $F\in Q$, $F$ is $\delta$-robust with respect to $\mathcal{T}^{j - 1}[q]$.
    \item $|\mathcal{T}^{|P|}| \geq  \frac{2\ve}{3} n!$. 
\end{enumerate}
\end{theorem}

 Before starting the proof, we set up some notation and claims that will be useful. In what follows, we fix $\F\sub \wB_n$ as in the theorem statement together with the large constant
\[\Delta := 12 |P| + q + 2,\] 
and for later convenience we define
\[K := \Delta q^2 \left( \frac{\Delta}{\Delta - 2} \right)^{|P|}.\]

  Starting with $\mathcal{T}^0$, we will build our subsets $\T^{j}$ as follows. Suppose $\T^{j-1}$ has already been defined for some $1\le j\le |P|$, and for ease of notation let $\M^{j-1}=\T^{j-1}[q]$.
For each $i=1,\dots, q$ and for every $\chi \in \mathcal{C}$, let $B^{j - 1}(\chi, i, \uparrow)$ be the set of members $F$ in $\mathcal{T}^{ j - 1}(\chi)$ such that $F$ is the $i$th member of some $q$-chain in $\mathcal{M}^{j - 1}(\chi)$ and $F$ is $(i, \delta)$-upper bad to $\mathcal{M}^{j - 1}$. Let  $B^{j - 1}(\chi, i, \downarrow)$ be the set of members $F$ in $\mathcal{T}^{ j - 1}(\chi)$ such that $F$ is $i$th member of some $q$-chain in $\mathcal{M}^{j - 1}(\chi)$ and $F$ is $(i, \delta)$-lower bad with respect to $\mathcal{M}^{j - 1}$. For convenience we denote the union of these sets by $$B^{j-1}(\chi, \downarrow) =  \bigcup_{1 \leq i \leq q}  B^{j - 1}(\chi,i, \downarrow),$$ $$B^{j-1}(\chi, \uparrow) =  \bigcup_{1 \leq i \leq q}  B^{j - 1}(\chi, i, \uparrow),$$
$$B^{j - 1} (\chi) = B^{j - 1}(\chi, \downarrow) \cup B^{ j -1}(\chi, \uparrow).$$

We now classify our chains $\chi \in \C$ based on whether they contain a relatively large number of bad members or not.  To this end we define
\[\mathcal{C}_1^j(\downarrow) = \left\{ \chi : |B^{j-1}(\chi, \downarrow)|>|\mathcal{T}^{j-1}(\chi)|/\Delta\right\},\]
\[\mathcal{C}_1^j(\uparrow) = \left\{ \chi : |B^{j-1}(\chi, \uparrow)|>|\mathcal{T}^{j-1}(\chi)|/\Delta\right\},\]
\[\mathcal{C}_2^j=\mathcal{C} - \mathcal{C}_1^j(\downarrow)  - \mathcal{C}_1^j(\uparrow).\] 
Let
\[\T^j=\{(\chi,F):\chi \in \C_2^j,\ F\in \T^{j-1}(\chi)-B^{j-1}(\chi)\}.\]
In other words, to form $\T^j$ from $\T^{j-1}$, we remove all $(\chi,F)\in \T^{j-1}$
from each $\chi\in \C^j_1(\downarrow)\cup \C_1^j(\uparrow)$ (i.e.\ from those $\chi$ with a large number of bad members), and for each $\chi\in \C^j_2$ we remove those $(\chi,F)$
where $F$ is bad.  In particular, we record the following immediate consequence of the definition of $\C_2^j$.
\begin{lemma}\label{C2Bound}
    If $\chi\in \C_2^j$ for some $j$, then
    \[|\T^j(\chi)|\ge \left(1-\frac{2}{\Delta}\right)|\T^{j-1}(\chi)|.\]
\end{lemma}

It remains to analyze our process for constructing $\T^j$.  For this, we develop some properties of full chains $\chi$ in $\mathcal{C}_1^j(\downarrow)$ and we will
then use these properties to show  $\sum_{\chi\in \mathcal{C}_1^j(\downarrow)}|{\mathcal{T}^{j - 1}(\chi)}|$
is relatively small; the situation for $\mathcal{C}_1^j(\uparrow)$ is similar.  For this result, given  an index $i$ and a full chain $\chi$, let $\chi_\F(i)$ be the $i$th member of $\F \cap \chi$. If no such member exists, by convention we let this denote the empty set.

\begin{lemma}\label{lem:function}
    For every $j$, there exists a function $\vec{b}$ from $ \chi \in \C^j_1(\downarrow) $ to increasing sequences of integers of even length of the form $(b_1, b_1', b_2, b_2', \dots )$ with the following properties: 
    \begin{enumerate}
        \item There exists an $i\in [q]$ such that for all $1 \leq \ell \leq \frac{|\vec{b}(\chi)|}{2}$, we have $\chi_{\F}(b_\ell) \in B^{j - 1}(\chi, i, \downarrow)$. 
        \item For all $1 \leq \ell \leq \frac{|\vec{b}(\chi)|}{2}$, we have $\chi_{\F}(b_\ell' ) \in \W(\chi_{\F}(b_\ell), i,  \T^{j - 1}[q])$.
        \item If $|\vec{b}(\chi)| = 2m$, then $|\mathcal{T}^0(\chi)|\leq K m$ and $\vec{b}(\chi) \in \binom{[Km] } {2m}$.
        \item For any vector of increasing integers $\vec{c}$ of length $2m$, there are at most $q \delta^m n!$ chains $\chi$ satisfying $\vec{b}(\chi) = \vec{c}$. 
    \end{enumerate}
\end{lemma}
\begin{proof}[Proof of Lemma~\ref{lem:function}]
    We begin by explicitly defining $\vec{b}(\chi)$ for each $\chi\in \mathcal{C}_1^j(\downarrow)$, and will refer to this vector as  \emph{the lower-bad profile} for $\chi$ relative to $\mathcal{M}^{j - 1}$ . 
    
    Fix some $\chi\in \mathcal{C}_1^j(\downarrow)$.  Let $F_1 \supset F_2 \dots \supset F_t$ be the members of $\mathcal{T}^0(\chi)$. Let $F_{a_1} \supset F_{a_2} \supset \dots \supset F_{a_r}$ be the subsequence of $F_1, \dots F_t$ consisting of all of the members in $\mathcal{T}^{j - 1}(\chi)$. By definition of ${\mathcal C}_1^j(\downarrow)$, we have $|B^{j -1}(\chi, \downarrow) |  > |\mathcal{T}^{j-1}(\chi)|/\Delta$. By the pigeonhole principle, there exist some $i \leq q$ such that \[|B^{j-1}(\chi, i, \downarrow)| > |\mathcal{T}^{j-1}(\chi)| /q\Delta.\]

Fix such an $i$. We will now greedily build a tuple of integers $\vec{b}(\chi)=(b_1, b_1', b_2, b_2',\dots,  b_m, b_m')$ such that for all $1 \leq \ell \leq m$, $F_{b_\ell}$ is $(i, \delta)$-lower bad relative to $\M^{j-1}$ and  $F_{b_\ell'} \in \W(F_{b_{\ell}}, i, \T^{ j -1}[q])$,   as follows.

Let $d$ be the smallest integer such that $F_{a_d}\in B^{j-1}(\chi, i, \downarrow)$; such a $d$ exists because $B^{j-1}(\chi, i, \downarrow) \neq \emptyset$. By definition, $F_{a_d}$ is the $i$th member of some marked chain $(\chi, Q)$ in $\mathcal{M}^{j - 1}(\chi)$.

In particular, this implies there are at least $q-i$ additional members of $\mathcal{T}^{j-1}(\chi)$ on $\chi$ below $F_{a_d}$. Because $F_{a_d}\in B^{j-1}(\chi,i,\downarrow)$, it is the $i$th member of some $q$-marked chain in $\mathcal{M}^{j - 1}(\chi) = \binom{\mathcal{T}^{ j - 1}(\chi)}{q}$, and in particular there exists at least $i - 1$ members which come before it, and at least $q - i$ members coming after in inside $\T^{ j - 1}(\chi)$. By property $(b)$ of witness sets we have then  $\{F_{a_{d +1}}, \dots F_{a_{d+ q - i}}\}\cap \W(F_{a_d}, i, \T^{( j - 1)}[q])\neq\emptyset$. Let $b_1 = a_d$ and let $b_1'$ be the index of the any member  $\{F_{a_{d +1}}, \dots F_{a_{d+ q - i}}\}\cap \W(F_{a_d},i, \T^{( j - 1)}[q])$.

Now let $d'$ be the smallest integer such that $F_{a_{d'}} \in B^{j-1}(\chi, i, \downarrow)$ and $a_{d'} > b_1'$ if it exists. Just as before, we are guaranteed one of  $F_{a_{d'+ 1}}, \dots F_{a_{d'+ q - i}}$ belongs to $\W(F_{a_{d'}}, i, \T^{( j - 1)}[q])$, and we let $b_2 = a_{d' }$ and $b_2'$ be the index of the any member of $\W(F_{a_{d'}}, i, \T^{( j - 1)}[q]) \cap \{ F_{a_{d'+ 1}}, \dots F_{a_{d'+ q - i}}\} $. We continue to repeat the process, e.g.\ by defining $d''$ to be the smallest integer such that $F_{a_{d''}} \in B^{j-1}(\chi, i, \downarrow)$ and $a_{d''} > b_2'$, if it exists, until no more choices remain. Note that the process goes on at least $\frac{1}{q - i} | B^{j - 1}(\chi, i, \downarrow)|$ many steps, as between $F_{b_i}$ and $F_{b_i'}$ there are at most  $q - i$ members of $B^{j - 1}(\chi, i, \downarrow)$. Furthermore, $\vec{b}(\chi) \subseteq [| \mathcal{T}^0(\chi)|]$.

\begin{claim}\label{sizeofint}
    If $\vec{b}(\chi)$ has length $2m$, then $|\mathcal{T}^0(\chi)|\leq K m$. 
\end{claim}
\begin{proof}
    
    By our assumption of $\chi\in {\mathcal C}_1^j(\downarrow)$, we must have $\chi\in {\mathcal C}^{\ell}_2$ for $\ell=0,1,\dots, j-1$ and $\chi\in {\mathcal C}_1^j(\downarrow)$.
   Therefore, for $1 \leq \ell \leq j - 1$, we know that $|\mathcal{T}^{\ell}(\chi) |\geq (1 - \frac{2}{\Delta})|\mathcal{T}^{\ell -1}(\chi)|$, and hence $|\mathcal{T}^0(\chi)| \leq (\frac{\Delta}{\Delta - 2})^{j}|\mathcal{T}^{j - 1}(\chi)|$. 
   By definition of $\chi \in \mathcal{C}_{1}^j(\downarrow)$, we know that  $|B^{j - 1}(\chi, \downarrow)| \geq |{\mathcal T}^{j-1}(\chi)|/\Delta$. By our choice of $i$,  $|B^{j - 1}(\chi, i, \downarrow) |\geq \frac{1}{q} |B^{j - 1}(\chi, \downarrow)|$. 
    By our earlier observation, we have that $m \geq \frac{1}{q - i} | B^{j - 1}(\chi, i, \downarrow)|$. 
    It therefore follows that $$|\mathcal{T}^0(\chi)|\leq  \left(\frac{\Delta}{\Delta-2}\right)^{|P|} \Delta |B^{j - 1}(\chi, \downarrow)|  \leq  \Delta q^2 \left( \frac{\Delta}{\Delta - 2} \right)^{|P|} m\leq K m$$.
\end{proof}
\begin{claim}\label{badtuples}
Let $\chi\in \mathcal{C}_1^j(\downarrow)$. Suppose $\vec{b}(\chi)$ has length $2m$. Then $\vec{b}(\chi) \in \binom{[Km] } {2m}$. 
\end{claim}

\begin{proof}
    Since $\vec{b}(\chi)$ has length $2m$,  by Claim~\ref{sizeofint}, $t:=|\mathcal{T}^0(\chi)| \leq K m$. Since the entries in $(b_1,b'_1,\dots, b_m, b'_m)$ are all inside
$\{1,2,\dots ,t\}\subseteq [Km]$, we have $\vec{b}(\chi) \in \binom{[Km] } {2m}$. \end{proof}

To prove the last part of the lemma, we will use the following technical result, where here we roughly think of $\mathcal{S}$ as the $F$ which are $(i,\delta)$-lower bad and $\W^*(F)=\W(F,i,\T^{j-1}[q])$.
\begin{claim}\label{numchains}
Let $\chi$ be a uniformly randomly chosen full chain in $\mathcal{B}_n$. Let $\mathcal{S} \subseteq \F$ be such that for every $F \in \mathcal{S}$, there is a subfamily $\W^*(F) \subseteq \B_n$ such that $\prob[\W^*(F) \cap \chi \neq \emptyset | F \in \chi] \leq \delta$. Then given an increasing tuple $\Vec{c} =(c_1, c_1', c_2, c_2' \dots, c_m, c_m')$, the probability that $\chi$ satisfies $\chi_{\F}(c_\ell) \in \mathcal{S}$ and $\chi_\F(c_\ell') \in \W^*(\chi_{\F}(c_i))$ for all $1 \leq \ell \leq m$ is less than $\delta^m$. 
\end{claim}

\begin{proof}
Let $A_\ell$ be the event that $\chi_{\F}(c_\ell) \in \mathcal{S}$ and $\chi_\F(c_\ell') \in \W^*(\chi_{\F}(b_\ell))$. Then,
\begin{align*}
    \prob[A_1\cap A_2 \cap \dots \cap  A_m] &= \prob[A_1] \prob[A_2 | A_1] \dots \prob[A_m  | A_1\cap A_2 \cap \dots\cap  A_{m - 1}].
\end{align*}
Note that
$$\prob [\chi_\F(b_\ell') \in \W^*(F) | (\chi_{\F} (c_\ell) = F)\cap  A_1\cap  A_2\cap \dots \cap A_{ \ell - 1} ] = \prob [\chi_\F(b_\ell') \in \W^*(F) | \chi_{\F} (c_\ell) = F   ] ,$$
 since conditioned on $\chi_{\F} (c_\ell) = F $ the events $\chi_\F(c_i') \in \W^*(F)$  and $A_1\cap A_2 \cap \dots A_{\ell-1}$ are independent.

Thus, by definition of $\W^*(F)$,  we have 
\begin{align*}
    \prob [\chi_\F(b_\ell') \in \W^*(F) | \chi_{\F} (c_i) = F)  ] &\leq \prob[ \chi \cap \W^*(F) \neq \emptyset | \chi_{\F} (c_\ell) = F)]\\
    &\leq \delta
\end{align*}
Observe that $A_\ell$ is exactly the union over the events $\chi_{\F} (c_\ell) = F$ over $F \in \mathcal{S}$ and $\chi_{\F}(c_\ell') \in \W^*(F)$. Since these events are disjoint, we have that the following series of inequalities hold:
\begin{align*}
    \prob[A_\ell| A_1 \cap A_2 \cap \dots \cap A_{ \ell - 1}] &= \sum_{F \in \mathcal{S}} \Big( \prob[\chi_{\F} (c_\ell) = F | A_1\cap A_2 \cap \dots A_{ \ell - 1}] \\ & \cdot \prob [\chi_\F(c_\ell') \in \W^*(F) | (\chi_{\F} (c_i) = F)\cap A_1\cap  A_2\cap \dots A_{ \ell - 1} ] \Big)\\ &\leq  \sum_{F \in \mathcal{S}} \prob[\chi_{\F} (c_\ell) = F | A_1 \cap A_2 \cap \dots A_{ \ell - 1}] \delta \\
    &\leq \delta. 
\end{align*}

Therefore, 
\begin{align*}
    \prob[A_1\cap A_2 \dots A_m] &= \prob[A_1] \prob[A_1 | A_2] \dots \prob[A_m  | A_1\cap A_2 \cap \dots A_{m - 1}] \\
    &\leq \delta^m.
\end{align*} 
\end{proof}

We now prove the last part of the lemma. Note that for each chain $\chi \in\C_1^j(\downarrow)$ which satisfies $\vec{b}(\chi)=\vec{c}$, all of the sets of the form $\chi_{\F}(c_\ell)$ are $(i, \delta)$-lower bad for some $i\in [q]$ with $\chi_{\F}(c'_\ell)$ in the corresponding set $\W(F,i,\T^{j-1}[q])$. Thus, applying Lemma~\ref{numchains} with $\mathcal{S}=\bigcup_{\chi\in \C_1^j(\chi)} B^{j-1}(\chi,i,\downarrow)$  the set of $(i, \delta)$-lower bad members, and taking $\W^*(F) = \W(F, i, \T^{j - 1}[q])$, we have that $\chi$ satisfies the conclusion of Lemma~\ref{numchains}. Thus, there are no more than $\delta^m n!$ such chains. Summing over all $i$, there are at most $q \delta^m n!$ many chains satisfying $\vec{b}(\chi) = \vec{c}$.

\end{proof}

\begin{lemma}\label{downweight} For each $1\leq j \leq |P|$,
 $$  \sum_{\chi\in \mathcal{C}_1^j(\downarrow)}{|\mathcal{T}^{j - 1}(\chi)|} \leq \frac{\ve}{18  |P|} n!. $$
\end{lemma}
\begin{proof}
Let $\vec{b}$ be the function from Lemma~\ref{lem:function}. 

Recall that Lemma~\ref{lem:function}.3  states that for all $m\in [n]$, each bad profile of length $2m$ is a member of $\binom{[Km]}{2m}$. 
This together with part Lemma~\ref{lem:function}.4 implies 

 \begin{align*} \sum_{\chi\in \mathcal{C}_{1}^j(\downarrow)} |\mathcal{T}^{j - 1}(\chi)| & \leq \sum_{m=1}^n\sum_{\vec{c} \in \binom{[K m]}{2m}} \sum_{ \chi : \vec{b}(\chi) = \vec{c}} |\mathcal{T}^{j - 1}(\chi)|\\
 &\leq \sum_{m=1}^n\sum_{\vec{c} \in \binom{[K m]}{2m}} \sum_{\chi : \vec{b}(\chi) = \vec{c}} |\mathcal{T}^{0}(\chi)| \\
 &\leq \sum_{m=1}^n\sum_{\vec{c} \in \binom{[K m]}{2m}} \sum_{  \chi : \vec{b}(\chi) = \vec{c}} K m & \text{ (By Lemma~\ref{lem:function}.3)}\\
 &\leq \sum_{m=1}^n\sum_{\vec{b} \in \binom{[K m]}{2m}}   (\delta^m n!) q Km \\
 &\leq q \sum_{m=1}^n2^{Km} K^m \delta^m n!\\
 &\leq 2^{K + 1} K q \delta n!,
  \end{align*}
provided that $\delta < \frac{1}{2^{K + 2} K}$. 
This will be less than $\frac{\ve}{18  |P|}$ provided that $\delta < \frac{\ve}{18  |P|2^{ K + 1} K q}$. 
\end{proof}

By a similar argument, 
\begin{lemma}\label{upweight} For each $1\leq j \leq |P|$,
 $$  \sum_{\chi\in \mathcal{C}_1^j(\uparrow)}{\mathcal{T}^{j - 1}(\chi)} \leq \frac{\ve}{18 |P|} n!. $$
\end{lemma}

We are now in a position to prove our main theorem in this section.

\begin{proof}[Proof of Theorem \ref{thm:weakCleaning}]
    For each $j\in |P|$,
when building $\mathcal{T}^{j}$ from $\T^{j-1}$ we made sure to remove all the members that are either $(i, \delta)$-lower-bad or $(i, \delta)$-upper-bad relative to $\T^{j-1}[q]$ for any $i\in [q]$. So every member $F$ on a $q$-chain in $\mathcal{T}^j[q]$ is $\delta$-robust with respect to $\mathcal{T}^{j-1}[q]$.

We will inductively prove for all $0\le j\le |P|$ that
\[\T^j \text{ is $q$-strong and }|\T^j|\ge \left(1-\frac{1}{3|P|}\right)^{j}|\T^0|,\]
The base case holds by our conditions.
Assume that we have proven the statements for all $\ell\leq j-1$, which in particular means $|\T^{j-1}|\ge \frac{2}{3}\ve n!$ since $|\T^0|\ge \ve n!$ by hypothesis.
Consider any $j\in [|P|]$. Let $\chi$ be a full chain in $\B_n$
with $\T^j(\chi)\neq \emptyset$. Then by our process, this means that $\chi\in \C^\ell_2$ for
each $\ell=1,\dots, j-1$. If $\T^j(\chi)=\T^{j-1}(\chi)$, then by induction hypothesis, $|\T^j(\chi)|\geq q$.
Hence, we may assume that $\T^{j-1}(\chi)$ contains at least one bad member. If $|\T^{j-1}(\chi)|<\Delta$,
then $\chi$ would have been in $\C^j_1$, a contradiction. So $|\T^{j-1}(\chi)|\geq \Delta$. By definition of
$\chi\in \C^j_2$, we have by \Cref{C2Bound} that \[|\T^j(\chi)|\geq \left(1-\frac{2}{\Delta}\right)|\T^{j-1}(\chi)|\geq \Delta-2\geq q.\] 

Thus, $\T^{(j)}$ is $q$-strong. Furthermore,  
\begin{align*}
 |\T^j|&= \sum_{\chi \in \mathcal{B}_n}|\mathcal{T}^j(\chi)| \geq \sum_{\chi \in \mathcal{B}_n}|\mathcal{T}^{j - 1}(\chi)| - \sum_{\chi \in \mathcal{C}_2^j} |B^{j - 1}(\chi) |- \sum_{\chi \in C_1^j(\downarrow)}|\mathcal{T}^{j - 1}(\chi)| - \sum_{\chi \in C_1^j(\uparrow)}|\mathcal{T}^{j - 1}(\chi) |\\
    &\geq \sum_{\chi \in \mathcal{B}_n}|\mathcal{T}^{j - 1}(\chi)| - \frac{2}{\Delta}  \sum_{\chi \in \mathcal{B}_n}|\mathcal{T}^{j - 1}(\chi)| - \frac{\ve}{9 |P|} n! \quad \text{ (By Lemma~\ref{downweight} and Lemma~\ref{upweight})}\\
    &\geq|\mathcal{T}^{ j - 1}| - \frac{1}{3  |P|}|\mathcal{T}^{ j - 1}|,
\end{align*}

where in this last step used $\Delta\ge 12|P|$ and that inductively $\frac{2}{3} \ve n!\leq |\T^{j-1}|$.

Thus, we have for all $1 \leq j \leq |P|$, $$|\mathcal{T}^{j}| \geq \left(1 - \frac{1}{3|P|}\right)^{|P|} |\mathcal{T}^{0}|  \geq \frac{2}{3} |\mathcal{T}^{0}|\ge \frac{2 \ve}{3}n!, $$
completing the proof. 
\end{proof}

Finally, we combine all of the results we have established up to this point into a single statement. Also note that the second part of this result is not needed for our proofs, but we include it in the statement since it adds no extra difficulty to the proof and may be useful for future applications. 

\begin{theorem}\label{thm:cleaning}
    For all integers $q\ge 1$ and reals $\ve>0$, there exists some $\del>0$ such that the following holds.  If $\F\sub\wB_n$ is a family with $\mu(\F)\ge q-1+\ve$, then there exists a nested sequence of $q$-marked chains $\M^{|P|}\sub \M^{|P|-1}\sub \cdots \sub \M^0$ such that for all $j\in [|P|]$ we have:
    \begin{itemize}
        \item For each $(\chi,Q)\in \M^j$ and $F\in Q$, we have that $F$ is $\delta$-robust with respect to $\M^{j-1}$.
        \item For each $i\in [q],$ we have 
        \[|\mathcal{L}^i(\M^j)|\ge \frac{2\ve}{3q} \min_{F\in \F}{n\choose |F|}.\]
    \end{itemize}
\end{theorem}

\begin{proof}
    By \Cref{size-to-weight} there exists a $q$-strong 1-marked family $\T^0$ from $\F$ with $|\T^0|\ge \ve n!$.  We can thus apply \Cref{thm:weakCleaning} with this $\T^0$ to obtain a nested sequence $\T^{|P|}\sub \cdots \sub \T^0$ of $q$-strong 1-marked chain families such that property 1 holds for $\M^j:=\T^j[q]$.  Moreover, \Cref{thm:weakCleaning} guarantees $|\T^j|\ge |\T^{|P|}|\ge \frac{2\ve }{3} n!$, so applying  \Cref{lem:translatingWeight} to $\T^j$ gives property 2, completing the proof.
\end{proof}

\subsection{Tools for Induced Posets}\label{sec:tools}

Throughout this subsection, let $P$ be a fixed tree poset of height $k \leq q$.  Given a set $F \in \B_n$, we let $U(F) = \{ S \in \B_n : S \supseteq F\}$ and $D(F) = \{ S \in \B_n : S \subseteq F\}$. If $\mathcal{S}$ is a subfamily of $B_n$ we let 
$$U(\mathcal{S}) = \bigcup_{S \in \mathcal{S}} U(S) \text{ and } D(\mathcal{S}) = \bigcup_{S \in \mathcal{S}} D(S)$$

Let $$\mathrm{Comp}(\mathcal{S})= U(\mathcal{S}) \cup D(\mathcal{S}).$$ Note that $\mathrm{Comp}(\mathcal{S})$ is the set of members of $\B_n$ which comparable to some member of $\mathcal{S}$. 

Furthermore, given a set $F$ and family $\mathcal{S}$ such that $\mathcal{S} \cap U(F) = \emptyset$, we set: 

$$D^*(F, \mathcal{S}) = (D(F) \setminus \{F\}) \cap \mathrm{Comp}(\mathcal{S}) \cap \wB_n$$

and if $\mathcal{S} \cap D(F) = \emptyset$

$$U^*(F, \mathcal{S}) = (U(F) \setminus \{F\}) \cap \mathrm{Comp}(\mathcal{S}) \cap \wB_n$$

We call these sets \textit{the forbidden neighborhood} of $\mathcal{S}$ with respect to $F$. 
We note that these notions are needed only for the induced part of our proof.
To that end, we need two lemmas from \cite{BJ}.

\begin{lemma}\label{downset}[Lemma 3.1 in \cite{BJ}]
    Let $F \in \wB_n, \mathcal{S} \subseteq \wB_n$ where $\mathcal{S} \cap U(F) = \emptyset$ and $|\mathcal{S}| \leq n / 6$. Let $\chi$ be a uniformly random full chain in $\C$. Then, 
    $$\prob[ \chi \cap D^*(F, \mathcal{S} ) \neq \emptyset | F \in \chi] \leq \frac{39 |\mathcal{S} | \sqrt{n \ln n}}{n}$$
\end{lemma}

\begin{lemma}\label{upset}[Lemma 3.2 in \cite{BJ}]
        Let $F \in \wB_n, \mathcal{S} \subseteq \wB_n$ where $\mathcal{S} \cap D(F) = \emptyset$ and $|\mathcal{S}| \leq n / 6$. Let $\chi$ be a uniformly random full chain in $\C$. Then, 
    $$\prob[ \chi \cap U^*(F, \mathcal{S} ) \neq \emptyset | F \in \chi] \leq \frac{39 |\mathcal{S} | \sqrt{n \ln n}}{n}$$
\end{lemma}

Let $\M$ be a family of $q$-marked chains. Let $i\in [q]$, $F\in \B_n$. Recall that for each $\chi\in \C$, we let
$\mathcal{M}(\chi, F,i)$ denote the set of all $(\chi, Q)\in \mathcal{M}$ such that $F$ is the $i$-th member of $Q$, and we let $\mathcal{M}(F,i)=\bigcup_{\chi}{\mathcal{M}(\chi, F,i)}$.

We say a family $\F \subseteq \B_n$ is {\it $\ell$-gapped} if for every $F, G \in \F$, with $F \subsetneq G$,  $|G - F| \geq \ell$. For $q \geq 1$, we say that a $q$-marked chain family $\M$ with markers from $\F$ is  $\ell$-gapped if $\F$ is $\ell$-gapped.

\begin{definition}\label{badv2}
    Fix $\gamma > 0$. Let $i,s\in [q]$ with $i <  s$.
We say that $F \in \mathcal{L}^i(\M)$ is $(i, s, \gamma)$-bad with respect to an $\ell$-gapped family $\M$, if there exist two families of sets $\mathcal{W}_1, \mathcal{W}_2 \subseteq \wB_n$ such that the following conditions hold. 
\begin{enumerate}
    \item $\mathcal{W}_1 \cap U(F) = \emptyset$ and $|\mathcal{W}_1| \leq |P|.$
    \item $|\mathcal{W}_2 | \leq \gamma n^{\ell (s - i) }.$
        \item For each $(\chi, Q)\in \M(F, i)$, either $Q\cap D^*(F, \mathcal{W}_1)\neq \emptyset$ or the $s$th member of $Q$ is in $\mathcal{W}_2$. 
\end{enumerate}

\end{definition}

\begin{definition}\label{badv2b}
    Similarly for $i,s\in [q]$ with $i >  s$, we say that $F \in \mathcal{L}^i(\M)$ is $(i, s, \gamma)$-bad with respect to an $\ell$-gapped family $\M$ if there exists two families of sets $\mathcal{W}_1, \mathcal{W}_2 \subseteq \wB_n$ such that the following conditions hold. 
\begin{enumerate}
     \item $\mathcal{W}_1 \cap D(F) = \emptyset$ and $|\mathcal{W}_1| \leq |P|.$
         \item $|\mathcal{W}_2 | \leq \gamma n^{\ell ( i - s)}.$
    \item For each $(\chi, Q)\in \M(F, i)$, either $Q\cap U^*(F, \mathcal{W}_1)\neq \emptyset$ or the $s$th member of $Q$ is in $\mathcal{W}_2$. 
\end{enumerate}

\end{definition}

This next lemma connects this newly defined notion of badness with the badness notion defined in the previous section. 

\begin{lemma}\label{lem:badequiv}
For every $\delta > 0$, and every $\ell, q$ positive integers, there exists $\gamma = \gamma(\delta, \ell, q)$ such that the following holds for every $i,s\in [q]$ such that $i\neq s$. 
    Let $\M$ be a $\ell$-gapped family of $q$-marked chains with markers from $\wB_n$ and let $F \in \mathcal{L}^i(\M)$. 
    Suppose $F$ is $(i, s, \gamma)$-bad with respect to $\M$ and that $n$ is sufficiently large. 
    \begin{itemize}
        \item If $i<s$, then $F$ is $(i, \delta)$-lower bad with respect to $\M$.
        \item If $i>s$, then $F$ is $(i, \delta)$-upper bad with respect to $\M$.
    \end{itemize}
\end{lemma}

\begin{proof}
    We will only prove the case $i<s$, the other case is analogous.  Because $F$ is $(i,s,\gamma)$-bad, there exist two sets $\mathcal{W}_1,\mathcal{W}_2$ satisfying the conditions in the definition and we will choose $\mathcal{W}_2$ to be minimal i.e. no $\mathcal{W}'_2 \subsetneq\mathcal{W}_2$ satisfies condition three with $\mathcal{W}_1$.  Observe that this minimality of $\mathcal{W}_2$ implies that every $D\in \mathcal{W}_2$ is in the $s$th position of some $q$ chain in $\M(F, i)$.
    
     We will show that 
    $F$ is $(i, \delta)$-lower bound with respect to $\M$ with $W:=\mathcal{W}_2 \cup D^*(F, \mathcal{W}_1)$ being an $(i,\delta)$-lower witness. Since $\M$ is $\ell$-gapped and for all $D\in \mathcal{W}_2$, $D$ is in the $s$th position of some $q$-chain in $\M(F, i)$,
    we have $|F-D|\geq \ell(s-i)$. Since $n$ is sufficiently large and $F,D\in \wB_n$, $|F-D|\leq n/6$.
    Thus for a uniformly chosen random full chain $\chi\in \C$ we have
    \begin{align*}
        \prob[\chi \cap \mathcal{W}_2 \neq \emptyset | F \in \chi] &\leq \sum_{D \in \mathcal{W}_2} \frac{1}{\binom{|F|}{|F - D|}} & \text{ by~(\ref{eqn:lubell})}\\
        &\leq |\mathcal{W}_2| \frac{1}{\binom{n / 3}{ \ell(s - i)}}\\
        &\leq \gamma n^{\ell (s - i)} (3\ell(s - i))^{\ell (s - i)} n^{ - \ell(s - i)}\\
        &\leq \gamma   (3\ell(s - i))^{\ell (s - i)}\\
         &\leq \frac{\delta}{2},
    \end{align*} by the choice of $\gamma = \frac{\delta}{2 (3 \ell q)^{\ell q}}$. 
    On the other hand by Proposition~\ref{downset}, 
    $$\prob[\chi \cap D^*(F, \mathcal{W}_1) \neq \emptyset | F \in \chi] \leq \frac{39 |P| \sqrt{n \ln n }}{n} \leq \frac{\delta}{2}.$$

    Thus, by the union bound, 
    $$\prob[ \chi \cap \W \neq \emptyset | F \in \chi] \leq \delta.$$
    This implies that $F$ is $(i, \delta)$-lower bad, because by definition, every $(\chi, Q)\in \M(F, i)$ intersects $W$. 
\end{proof}

\section{Embedding Tree Posets}\label{sec:embedding}
In this section we use the tools developed in the previous section to prove our two main supersaturation results: \Cref{thm:main} and a balanced supersaturation version of \Cref{cor:supersat} stated formally as \Cref{thm:balanced}.  Both of these proofs will follow essentially the same scheme for embedding a tree poset $P$  into a large family $\F$ one member at a time by using \Cref{thm:cleaning} to show that at every step, we always have many choices for how to embed the next member of $P$ even when we forbid some number of ``bad'' choices in $\F$.

\subsection{Embedding Process}
We establish the following general embedding result  \Cref{thm:genl} using \Cref{thm:cleaning}. It basically states that given some sufficiently sparse forbidden set $\Gamma\subseteq 2^{\F}$, we can grow many induced copies of $P$ not in $\Gamma$. For the non-balanced supersaturation we actually do no need to have such a forbidden set (i.e. $\Gamma = \emptyset$) but for the balanced supersaturation result, $\Gamma$ will be the upward closure of all the sets which are already contained in too many induced copies (i.e. saturated sets) of $P$ in the current collection.


\begin{definition}\label{def:blockingfamily}

Given an $\ell$-gapped family $\F$, we say a collection $\Gamma \subseteq 2^{\F}$ is a \textit{$(\gamma,\ell,  \F)$-bounded family} if it satisfies the two properties below
\begin{enumerate}
    \item We have $\{F\}\notin \Gamma$ for all $F\in \F$.
    \item If $\D\not\in\Gamma$, then the number of $F \in \F$ such that $\{F\} \cup \D\in \Gamma$ is less than $\gamma n^{\ell}$. 
\end{enumerate}

\end{definition}
\begin{theorem}\label{thm:genl}
    Let $P$ be a tree poset of height $k$. Let $q \geq k$ and $\ell$ be positive integers and $\ve > 0$ then there exists a constant $\gamma = \gamma(\ve, q, \ell)$ such that the following holds for any $v \in P$ and for any $\rank:P\to [q]$ a poset homomorphism. Let $\F \subseteq \wB_n$ be an $\ell$-gapped family and has Lubell weight $\mu(F) \geq q - 1 + \ve$ and  $\Gamma$ a $(\frac{\gamma}{2}, \ell, \F)$- bounded family. Then, there is $F \in \F$ such that the number of injective induced homomorphisms $\varphi : P \rightarrow \F'$ satisfying $\varphi(P) \not \in \Gamma$ and $\varphi(v) = F$ is at least
 \[\left( \frac{\gamma}{2} \right)^{|P| - 1}\left(\prod_{xy\in H(P)}n^{\ell |\rank(y)-\rank(x)|} \right).\]

\end{theorem}

\begin{proof}
    Because $P$ is a tree poset, there exists an ordering $v_1, v_2, \dots v_{|P|}$ of the members of $P$ with $v_1 = v$ such that every member $v_j$ with $j\ge 2$ has exactly one neighbor $v_{j'}$ with $j'<j$ in the Hasse diagram of $P$, and we call this $v_{j'}$ the \textit{parent} of $v_j$.

Apply Theorem~\ref{thm:cleaning} to $\F$ with $q, \ve $ to find a $\delta = \delta(\ve, q) > 0$ and a nested sequence of $q$-marked chains $\M^{|P|}\sub \M^{|P|-1}\sub \cdots \sub \M^0$ with markers from $\F$ such that for all $j\in [|P|]$ we have that for each $(\chi,Q)\in \M^j$ and $F\in Q$, we have that $F$ is $\delta$-robust with respect to $\M^{j-1}$.

Fix $\gamma$ be obtained from Lemma~\ref{lem:badequiv} applied with $q, \ell, \delta$. 

Fix $F_0\in\mathcal{L}^{\rank(v_1)}(\M^{|P|})$ and define our initial embedding $\varphi^1(v_1):=F_0$. To define the final embedding, we iteratively extend $\varphi^j: \{ v_1, v_2, \dots ,  v_j \} \rightarrow \F $ to $\varphi^{j+1}$ maintaining the following properties.

    \begin{enumerate}
        \item[$(C_1)$:] Let $H_j(P) :=  H(P)[\{v_1, v_2, \ldots, v_j\}]$, that is the Hasse diagram induced by the first $j$ elements of $P$. For all edges $v_av_b$ in $H_j(P)$, there exists a $(\chi, Q)$ in $\mathcal{M}^{|P| - j}$ with $\varphi^{j}(v_a)$ in the $\rank(v_a)$th position of $Q$ and $\varphi^{j}(v_b)$ in the $\rank(v_b)$th position of $Q$. 
        \item[$(C_2)$:] For all noncomparable pairs $v_a, v_b \in \{v_1, \dots, v_j\}^2$,  $\varphi^j(v_a)$ and $\varphi^j(v_b)$ are not comparable in $\B_n$. 
        \item[$(C_3)$:] $\{ \varphi^j(v_1), \varphi^j(v_2), \dots \varphi^j(v_j)\} \not \in \Gamma$.
    \end{enumerate}

     Note that these properties are satisfied for $\varphi^1$,  since $(C_1), (C_2)$ both hold vacuously and $(C_3)$ holds because $F_0 \not \in \Gamma$ by \Cref{def:blockingfamily} part 1.
      
    Let us emphasize that the choice of $F_0$ and  $(C_1)$ ensures $\varphi^j(v_c) \in \mathcal{L}^{\rank(v_c)}(\M^{|P| - j})$ for all $v_c \in \{v_1, \dots v_j\}$. This will be useful later.

    Furthermore, $(C_1), (C_2)$ ensures the poset induced by $\varphi^{j + 1}(v_1), \varphi^{j }(v_2), \dots, \varphi^{j }(v_{j})$ has Hasse diagram isomorphic to $H_{j}(P)$. As two posets are isomorphic if and only if their Hasse diagrams are isomorphic, by the end of the process we will obtain an induced copy $P'$ of $P$. Let us now show that indeed these partial embeddings are possible to construct, and count how many choices we have at each step.

 Now let $j\geq 1$ and suppose $(C_1), (C_2),$ and $(C_3)$ hold for $\varphi^j$. We want to extend $\varphi^j$ to a partial embedding $\varphi^{ j + 1}$ so that $(C_1)-(C_3)$ hold for $\varphi^{ j + 1}$. 
 
 For notational convenience let $x:=v_{j  + 1}$, and let $y$ denote the unique parent of $x$  which has already been embedded. We wish to find at least one way of embedding $x$. For further convenience, we only consider the case $y >_P x$, the other case being analogous. This in particular implies $\rank(y)<\rank(x)$.

 Recall the formal definition of $D^*(F, \Sa)$ from Section~\ref{sec:tools} for  a member $F\in \F$ and a family $\Sa$ which is simply the set of all members $S\in \Sa$ which are downsets of $F$ and are comparable to some set of $\Sa$.  Furthermore, let $\widehat{\M}^{|P| - j - 1} := \M^{|P| - j - 1}(\varphi^j(y), \rank(y))$, that is the family of all $q$-marked chains in $\M^{|P| - j - 1}$ which have $\varphi^j(y)$ in their $\rank(y)$th position. Let $\mathcal{P}_j = \{ \varphi^j(z) : z \not \geq_P y , z \in \{ v_1, v_2, \dots, v_{j} \}\}$, that is the set of all images of currently embedded elements of $P$ which are either below $y$ or incomparable with $y$, and thus not cannot be comparable with $\varphi^{j + 1}(x)$ because in the embedding $\varphi^j$, $x$ has the unique neighbor $y$. 
 
 The following set  encodes the choices for $x$ that would preserve $\varphi^{j+1}$ being an induced homomorphism. Now let

$$\mathcal{A} =\mathcal{L}^{\rank(x)}(\widehat{\M}^{|P| - j - 1}) - D^*(\varphi^j(y), \mathcal{P}_j ),$$
i.e.\ this is the set of $F$ which are not in the forbidden neighborhood of $\mathcal{P}_j$ and which are in the $\rank(x)$th position of some $q$-marked chain of $\M^{|P| - j - 1}$ which has $\varphi^j(y)$ in the $\rank(y)$th position.

\begin{claim}\label{cl:largeA}
$$|\A| \geq \gamma n^{\ell (\rank(x) - \rank(y))}$$
\end{claim}
\begin{proof}[Proof of Claim~\ref{cl:largeA}] 
Suppose $|\mathcal{A}| < \gamma n^{\ell (\rank(x) - \rank(y))}$.
We derive a contradiction by showing this would imply $\varphi^j(y)$ is $(\rank(y), \rank(x), \gamma)$-bad (recall \Cref{badv2}) with
$\mathcal{W}_1=\mathcal{P}_j$ and $\mathcal{W}_2=\mathcal{A}$. This is a contradiction because Lemma~\ref{lem:badequiv} implies that $\varphi^j(y)$ is $(\rank(y), \delta)$-bad with respect to $\M^{|P| - j - 1}$. However, this cannot happen since  we constructed our marked chain families using \Cref{thm:cleaning} which guarantees that $\varphi^j(y)$ is $\delta$-robust with respect to $\M^{|P| - j - 1}$ (see Definition~\ref{deltarobust}).

 Note that by definition of $\mathcal{P}_j$ we have $\mathcal{P}_j\cap U(\varphi^j(y))=\emptyset$ and $|\mathcal{P}_j|\le |P|$, and by hypothesis we are assuming $|\mathcal{W}_2|\le \gamma n^{\ell(\rank(x) - \rank(y))}$. 
It thus remains to check the last condition in the definition of $(\rank(y),\rank(x),\gamma)$-badness, i.e.\ that for each $(\chi,Q)\in \widehat{M}^{|P| - j - 1}$, either $Q\cap D^*(\varphi^j(y),\mathcal{P}_j)\ne \emptyset$ or $Q$ contains a member from $\mathcal{A}$ in the $\rank(x)$th position.

Consider any $(\chi, Q)\in \widehat{\M}^{|P| - j - 1}$. If $Q\cap D^*(\varphi^j(y), \mathcal{P}_j) = \emptyset$, then the $\rank(x)$th member of $Q$ lies in $\mathcal{A}$ by definition of $\mathcal{A}$. Therefore, every $(\chi, Q) \in \widehat{\M}^{|P| - j - 1}$ has $Q$ either intersecting $D^*(\varphi^j(y), \mathcal{P}_j)$ or $\mathcal{A}$. 
\end{proof}

 Let $\A' = \A - \{F \in \F: \{F, \varphi^j(v_1), \varphi^{j}(v_2), \dots \varphi^{j}(v_j)\} \in \Gamma \}$. By $(C_3)$ and \Cref{def:blockingfamily}, $|\A'| \geq \frac{1}{2}  \gamma n^{\ell (\rank(x) - \rank(y))}$. Fix one such $F \in \A'$.  We claim that by embedding $x$ into any set in $\A'$, we will extend $\varphi^j$ to $\varphi^{j+1}$, satisfying all the desired properties $(C_1)-(C_3)$.  Indeed, define $\varphi^{j + 1}(v_{a}) = \varphi^{j} (v_{a})$ for all $a \in [j]$ and $\varphi^{ j + 1}(x) = F$. Let us now check that $\varphi^{j+1}$ satisfies $(C_1)-(C_3)$.

Since $\mathcal{M}^{ |P| - j} \subseteq \mathcal{M}^{|P| - j - 1}$, for all $a, b \in [j]$, $(C_1)$ and $(C_2)$ hold for all such pairs so we only need to check that these two conditions are satisfied for pairs of form $(x, v_a)$ for any $a\in [j]$.

By definition of $\A$, $\varphi^{j + 1}(x)$ is in the $\rank(x)$th position of some $(\chi, Q) \in \M^{|P| - j -1}$ which has $\varphi^{ j + 1}(y)$ in the $\rank(y)$th position.

 Let $a \in [j]$, and suppose $v_a$ is not comparable with $x$. We would like to show that $\varphi^{j+1}(v_a)$ and $F$ are incomparable. It is easy to see that $v_a \not \geq y$, so $\varphi^{j + 1}(v_a) = \varphi^j(v_a) \in \mathcal{P}_j$. Since $F \in D(\varphi^j(y))\setminus D^*(\varphi^j(y), \mathcal{P}_j)$ by our choice, $F$ is not comparable with any element in $\mathcal{P}_j$, in particular, $\varphi^{j + 1}(v_a)$.

To see $\varphi^{ j + 1}$ satisfies $(C_3)$, notice that $F$ was chosen so that $\{F, \varphi^j(v_1), \varphi^{j}(v_2), \dots \varphi^{j}(v_j)\} \not \in \Gamma$.

Finally, to check that the desired counting bound does hold, it is enough to observe that at every step while going from $\varphi^j$ to $\varphi^{j+1}$ starting at $j\geq 2$ we had at least  $\frac{1}{2}  \gamma n^{\ell (\rank(x) - \rank(y))}$ many choices to embed $x=v_{j+1}$, where  $y$ is the parent of $x$ among  already embedded vertices $v_1, v_2, \dots, v_{j}$.  Taking the product over all $j=2, \dots, |P|-1$ gives us the desired result.
\end{proof}

\subsection{Proof of Theorem~\ref{thm:main}}
We begin by proving an analogue of \Cref{thm:main} in the setting of Lubell weight and $\ell$-gapped families and deduce \Cref{thm:main} from it.  

\begin{theorem}\label{prop:embedding}
For every tree poset $P$ of height $k$, positive integers $q, \ell$  with $q \geq k$, and $\ve > 0$, there exists $\gamma=\gamma(\ve,q,\ell)$ such that the following holds for any $\rank : P \rightarrow [q]$ a poset homomorphism. Let $\F \subseteq \wB_n$ be an $\ell$-gapped family such that for all $\Sa \subseteq \F$ with $|\Sa| = N $ we have $\mu(\F - \Sa) \geq q - 1 + \ve$. Then, the number of induced copies of $P$ in $\F$ is at least
    \[\left(\frac{\gamma}{2}\right)^{|P| - 1}\left(\prod_{xy\in H(P)}n^{\ell |\rank(y)-\rank(x)|} \right) \cdot N.\]

\end{theorem}

\begin{proof}

Let $v$ be any vertex of $P$ and $\gamma = \gamma(\ve, q, \ell)$ the real number returned by \Cref{thm:genl} applied with $\ve, q, \ell$. 

Assume on the contrary that $\F$ contains less than  \[\left( \frac{\gamma}{2} \right)^{|P| - 1}\left(\prod_{xy\in H(P)}n^{\ell |\rank(y)-\rank(x)|} \right)\cdot N,\]
   induced copies of $P$.

Let $\F_{\mathrm{b}}$ be the set of elements in $\F$ which are embedded as $v$ in more than $\left(\frac{\gamma}{2} \right)^{|P| - 1} \prod_{xy\in H(P)}n^{\ell |\rank(y)-\rank(x)|}$ induced copies of $P$. 

Since by assumption, we know that 
$$|\F_{\mathrm{b}} | \cdot\left( \frac{ \gamma}{2} \right)^{|P| - 1} \prod_{xy\in H(P)}n^{\ell |\rank(y)-\rank(x)|} \leq \left(\frac{\gamma}{2} \right)^{|P| - 1} \left(\prod_{xy\in H(P)}n^{\ell |\rank(y)-\rank(x)|} \right)\cdot  N, $$
$$\implies |\F_{\mathrm{b}}| \leq N.$$

Let $\F' = \F - \F_{\mathrm{b}}$. Then, $\mu(\F') \geq q - 1 + \ve$ by assumption. 

Now, apply \Cref{thm:genl} to $\F'$ with $\Gamma = \emptyset, q, \ell, \ve$. Then, there is $F \in \F$ playing the role of $v$ in at least \[\left(\frac{\gamma}{2}\right)^{|P| - 1} \prod_{xy\in H(P)}n^{\ell |\rank(y)-\rank(x)|},\] induced copies of $P$.  This contradicts the construction of $\F'$, and so $\F$ contains at least 

\[\left( \frac{\gamma}{2} \right)^{|P| - 1}\left(\prod_{xy\in H(P)}n^{\ell |\rank(y)-\rank(x)|} \right) \cdot N,\]
   induced copies of $P$, completing the proof of the theorem.

\end{proof}

Lastly, we will need the following upper bound on $M^*(n,q,P)$ in terms of all poset homomorphisms  $\rank:P\rightarrow [q]$. Recall that $M^*(n,q,P)$  is the number of copies of $P$ in the $q$ middle levels of $\B_n$.

\begin{lemma}\label{lem:rankandmiddle}
    If $P$ is a tree poset and $q$ is any positive integer, then 
    \[M^*(n,q,P)\le \sum_{\rank}\prod_{xy\in H(P)}n^{|\rank(x)-\rank(y)|}\cdot {n\choose n/2},\]
    where the sum ranges over all poset homomorphisms $\rank:P\to [q]$.
\end{lemma}

\begin{proof}   
For ease of notation, let $\M_{n, q}$ be the middle $q$ layers of $\B_n$. Given a poset homomorphism $\rank:P\to [q]$, we say that a copy $P'$ of $P$ in $\M_{n,q}$ is of \emph{type-$r$} if for each $x\in P$, the member $F_x\in P'$ corresponding to $x$ lies in the $\rank(x)$th level of $\M_{n,q}$.  Observe that every copy of $P \in \M_{n,q}$ must be of type $\rank$ for some  $\rank$.  As such, to prove the result it will suffice to show that for all homomorphisms $\rank$, the number of induced copies of $P$ or type $\rank$ in $\M_{n,q}$ is at most
    \[\prod_{xy\in H(P)}n^{|\rank(x)-\rank(y)|}\cdot {n\choose n/2}.\]

    Let $x_1,\ldots,x_{|P|}$ be an ordering of $P$ such that for all $j\geq 2$ $x_j$ has a unique neighbour, called the \emph{parent} of $x_j$, among $x_1, \dots, x_{j-1}$ if we view $H(P)$ as an undirected graph. Note that such an ordering exists since $P$ is a tree poset.  We can identify the copies of $P$ in $\M_{n,q}$ of type-$\rank$ by tuples $(F_1,\ldots,F_{|P|})$ of members of $\M_{n,q}$ such that the map $f(x_i)=F_i$ defines a copy of $P$ of type $\rank$. So, it is enough to upper bound the number of such tuples.

    Since $F_1$ must be a member of the $\rank(x_1)$th layer of $\M_{n,q}$, the number of choices for $F_1$ is at most ${n\choose n/2}$.  Given that $F_1,\ldots,F_{j-1}$ have been selected, let $x_i$ be the parent of $x_j$.  If $x_i<x_j$, then $F_j$ must be a set containing $F_i$ together with $\rank(x_i)-\rank(x_j)$ additional elements from $[n]$, and the number of such sets is at most $n^{\rank(x_i)-\rank(x_j)}$.  Similarly if $x_j<x_i$ then the number of choices for $F_j$ is at most $n^{\rank(x_j)-\rank(x_i)}$.  Multiplying the number of choices for each step gives the total result, since each edge of $H(P)$ is counted exactly once by some $n^{|\rank(x_i)-\rank(x_j)|}$ term.
\end{proof}

With all of this established we can now complete the proof of our main supersaturation theorem. 

\begin{proof}[Proof of \Cref{thm:main}]  Using Chernoff bounds, it is standard~\cite{BJ, GL} to show that the number of sets $F\in \B_n$ with $||F|-n/2|>2\sqrt{n\ln{n}}$ is $o({n \choose n/2})$. Thus, we may assume $\F\sub \wB_n$.

Note that $\F$ is one-gapped and for every subfamily $\Sa$ of size $N=\frac{\ve}{2} \binom{n}{n / 2}$, we have that $\mu (\F - \Sa) \geq q - 1 + \frac{\ve}{2}$. Thus, applying \Cref{prop:embedding} with $\F, q, \ell = 1, \frac{\ve}{2}$ and taking the maximum over $\rank : P \rightarrow [q]$ we have that the number of induced copies of $P$ in $\F$ is at least
\[\Omega\left(\max_{\rank}\prod_{xy\in H(P)}n^{|\rank(y)-\rank(x)|}\cdot {n\choose n/2}\right)=\Omega\left(\sum_{\rank}\prod_{xy\in H(P)}n^{|\rank(y)-\rank(x)|}\cdot {n\choose n/2}\right)=\Omega\left(M^*(n,q,P)\right),\]
where the first equality holds since there are at most $|P|^q=O_{q,P}(1)$ possible poset homomorphisms $\rank:P\rightarrow [q]$, and the second does because of Lemma~\ref{lem:rankandmiddle}.
\end{proof}

\subsection{Balanced Supersaturation}
As mentioned above, a similar proof to that of \Cref{thm:main} can be used to improve \Cref{cor:supersat} to a balanced supersaturation result. To state this  formally, given a hypergraph $\mathcal{H}$ and a set of vertices $\D$, we define the degree $\deg_{\mathcal{H}}(\D)$ of $\D$ to be the number of edges of $\HH$ containing the set $\D$, and for an integer $j$ we define the maximum $j$-degree as
\[\Delta_{j}(\HH) := \max_{\D \subseteq V(\HH), |\D| = j} \deg_{\HH}(\D).\]
Given a collection $\HH$ of induced copies of a poset $P$ in a family $\F$, we can identify $\HH$ as a $|P|$-uniform hypergraph with vertex set $\F$ and with hyperedges consisting of sets of members of $\F$ which form an induced copy of $P$ in $\HH$.  With this we can now state our balanced supersaturation result, where here we recall that a family $\F$ is $\ell$-gapped if $|F-G|\ge \ell$ for all $F,G\in \F$ with $G\subsetneq F$.

\begin{theorem}\label{thm:balanced}
 For every tree poset $P$ with height $k$, real number $\ve > 0$, and integer $\ell$; there exists  $\delta = \delta (\ve, P, \ell)$ such that the following holds. Let  $n$ be sufficiently large and $\F \subseteq \widetilde{\mathcal{B}}_n$ satisfy $|\F| \geq (k - 1 + \ve) \binom{n}{ n / 2}$, and suppose $\F$ is $\ell$-gapped. Then there is a collection $\HH$ of induced copies of $P$ from $\F$ satisfying
    \begin{enumerate}
        \item $|\HH| \geq \delta^{|P|} n^{ \ell (|P| - 1) } \binom{n}{ n / 2},$
        \item $\Delta_{j}(\HH) \leq (\delta n^{\ell} )^{ |P| - j}$
    for all $1 \leq j \leq |P|$. \end{enumerate}
    
\end{theorem}
Note that this in particular implies \Cref{cor:supersat} since every $\F$ is 1-gapped. 
\begin{lemma}\label{lem:indstep}
    For every tree poset $P$ with height $k$ and $\ve > 0$, there exists a $\delta = \delta (\ve, \ell, P)$ such that the following holds. Let $n$ is sufficiently large and $\F \subseteq \mathcal{B}_n$ satisfy $|\F| \geq (k - 1 + \ve) \binom{n}{ n / 2}$ and suppose $\F$ is $\ell$-gapped.  If $\HH$ is a collection of copies of $P$ from $\F$ satisfying
    \begin{enumerate}
        \item[(P1).] $|\HH| \leq \delta^{|P|} n^{\ell ( |P|- 1)} \binom{n}{ n / 2},$
        \item[(P2).] For all $1\le  j \le |P|$, $\Delta_{j}(\HH) \leq (\delta n^{\ell} )^{ |P| - j},$
    \end{enumerate}
    then there exists an induced copy $P'$ of $P$ not in $\HH$ such that $\HH' = \HH \cup \{ P'\}$ satisfies $$\Delta_{j}(\HH') \leq (\delta n^{\ell})^{ |P| - j}$$
    for all $1 \leq j \leq |P|$. 
\end{lemma}

Observe that \Cref{thm:balanced} follows immediately from repeatedly applying Lemma~\ref{lem:indstep} until $|\HH| \geq \delta^{|P|} n^{\ell (|P| - 1)} \binom{n}{n /2}$, so it will suffice to prove this result.

\begin{proof}
    We say that $\D \subseteq \B_n$ is saturated if $1 \leq |\D| \leq |P|$ and $$\deg(\D) = \lfloor  (\delta n^{\ell} )^{ |P| - |\D|} \rfloor.$$
We say a subfamily $\K \subseteq \B_n$ with $|\K| \leq |P|$ is \textit{inadmissible}  if there exists a subfamily $\D \subseteq \K$ that is saturated. Otherwise, we say $\K$ is \textit{admissible}.

Observe that proving the lemma is equivalent to saying that there exists an admissible set $\K$ which forms an induced copy of $P$ which is not already in $\HH$.  To show this, we will start by removing from $\F$ any $F$ such that $\{F\}$ is saturated, as no such $F$ can ever be used in an admissible $\K$.  For this we observe the following. 
\begin{claim}\label{claim:few}
    If $\mathcal{F}_{\mathrm{sat}}\sub \F$ denotes the set  of $F\in \F$ such that $\{F\}$ is saturated, then $|\mathcal{F}_{\mathrm{sat}}|\le \frac{\ve}{2}{n\choose n/2}$.
\end{claim}
Here and throughout the proof we will make frequent use of the inequality $\lfloor(\delta n^\ell)^{|P|-i}\rfloor\ge \frac{1}{2} (\delta n^{\ell})^{|P|-i}$ for all $i\le |P|$, since $n$ is sufficiently large. 
\begin{proof}[Proof of \Cref{claim:few}]
    We first show that
    \[\lfloor (\delta n^{\ell})^{|P|-1}\rfloor \cdot |\mathcal{F}_{\mathrm{sat}}|\le \sum_{F\in \mathcal{F}_{\mathrm{sat}}} \deg(\{F\})\le |P|\cdot |\HH|.\]
    Indeed, the lower bound for the sum follows from the definition of what it means for $\{F\}$ to be saturated; the upper bound comes from the fact that the sum counts the number of pairs $(P',F)$ with $P'$ a copy of $P$ in $\HH$ and $F\in P'\cap \mathcal{F}_{\mathrm{sat}}$.

    Using the inequality above together with (P1) shows that
    \[\frac{1}{2}(\delta n^{\ell})^{|P|-1}\cdot |\mathcal{F}_{\mathrm{sat}}|\le  \lfloor (\delta n^{\ell})^{|P|-1}\rfloor \cdot |\mathcal{F}_{\mathrm{sat}}|\le |P|\cdot |\HH|\le |P| \cdot \delta^{|P|} n^{\ell (|P|-1)}{n\choose n/2},\]
    and rearranging gives the desired result, by choosing $\delta$ sufficiently small with respect to $\ve, P$.
\end{proof}
Define $\F':=\F-\mathcal{F}_{\mathrm{sat}}$. Note that 
\[|\F'|\ge (k-1+\ve/2){n\choose n/2}.\]

For every $\K\subseteq B_n$, let us define
$$\Z(\K) : = \{ F \in \F' : \{ F \} \cup \K \text{ is inadmissible}  \}.$$

 The intuition here is that if we have already partially built some set $\K$ to eventually be used in a copy of $P$, then $\Z(\K)$ represents the set of ``bad choices'' of $F$ that we could add to $\K$ to make it inadmissible.  A simple double counting argument shows that the number of such ``bad choices'' is relatively small.

\begin{claim}\label{lem:few-bad}
    For any admissible $\K \subseteq \B_n$ and $|\K| \leq |P|$, we have  $$|\Z(\K)| \leq 2^{|\K|} \cdot 2 \delta |P| n^{\ell}.$$  
\end{claim}
\begin{proof}[Proof of \Cref{lem:few-bad}]
For every subfamily $\D \subseteq \K$, let 
$$\Sa(\D) := \{ F \in \B_n :  \{ F\} \cup \D\text{ is saturated} \}. $$

Observe that since $\K$ is admissible, having $\{F\}\cup \K$ inadmissible implies that the subsets of $\{F\}\cup \K$ that are saturated must be of the form $\{F\}\cup \D$ for some $\D\sub \K$, i.e.\ we have $\Z(\K)=\bigcup_{\D\subseteq \K} \Sa(\D)$.  Moreover, because every $F \in \Z(\K)$$\sub \F'$ is unsaturated, we have that $\Sa(\emptyset)=\emptyset$.  In total, then we see that

\begin{equation} \label{bad-f}
\Z(\K)=\bigcup_{\D\subseteq \K, \D \neq \emptyset} \Sa(\D),
\end{equation} 
 and it now suffices to bound $|\Sa(\D)|$ for each  $\D \neq \emptyset$.

Fix some $\D\subseteq \K$ non-empty. Let $\mathcal{E}$  denote the set of pairs $(P',F)$ where $P'$ is a copy containing $\D$ and $F\in P'$ is arbitrary.
Let $\mathcal{E'} \subseteq \mathcal{E}$ be the set of tuples $(P', F) \in \mathcal{E}$ with the additional property that $F \in \Sa (\D)$. 
With this we see
\begin{equation}\sum_{F\in \Sa(\D)} \deg(\{F\}\cup \D)=|\mathcal{E}'|\le |\mathcal{E}|= |P|\deg(\D).\label{DF}\end{equation}

Since $\{F\}\cup \D$ is saturated for each $F\in \mathcal{S}(\D)$, we have $\deg(\{F\} \cup \D)=\lfloor (\delta n^{\ell})^{|P|-|\D|-1} \rfloor$. Since $\D\sub \K$ is unsaturated (because $\K$ is admissible), we have $\deg(\D)\leq (\delta n^{\ell})^{|P|-|\D|}$. Therefore, \eqref{DF} implies
$$|\Sa(\D)| \lfloor (\delta n^\ell)^{ |P| - |\D| - 1} \rfloor = \sum_{ F \in \Sa(\D)} \deg(\{F\} \cup \D) \leq |P| \deg(\D) \leq |P| (\delta n^{\ell})^{ |P| - |\D|} .$$

This implies

$$|\Sa(\D) |\leq 2 \delta |P| n^{\ell} .$$

This together with \eqref{bad-f} gives $|\Z(\K)|\leq 2^{|\K|} \cdot 2 \delta |P| n^{\ell}$ as desired.
\end{proof}

To complete the proof let $\gamma = \gamma(\ve, k, \ell)$ be derived from \Cref{thm:genl} applied with parameters $P, q = k$, $\ell,$ and $\frac{\ve}{2}$. We let $\Gamma$ be the set of all inadmissible sets. Note that by choosing $\delta$ sufficiently small with respect to $\gamma$, we may ensure that for all $\K$ admissible, $\Z(\K) \leq \frac{\gamma}{2} n^{\ell}$ by Claim~\ref{lem:few-bad}. This guarantees  $\Gamma$ to be $(\frac{\gamma}{2}, \ell, \F')$-bounded family. Now, applying \Cref{thm:genl} to $\F'$, $\Gamma,$ and any $v\in P$, we obtain an induced copy $P'$ of $P$ such that $\deg(\D) < \floor{(\delta n^{\ell} )^{|P| - |\D|}}$ for all $\D \subseteq P'$, and so $P' \not \in \HH$. Adding $P'$ to our collection $\HH$ preserves the desired maximum degree condition.
\end{proof}

\section{Using Balanced Supersaturation}\label{sec:containers}
In this section, we use our balanced supersaturation result \Cref{thm:balanced} together with the powerful method of hypergraph containers in order to prove \Cref{thm:counting} and \Cref{thm:randomturan}. In the next subsection, we prove our main container result \Cref{thm:major}, after which we use it together with standard arguments to conclude our main results.

\subsection{Hypergraph Containers}
 For a hypergraph $\HH$, we let $\I(\HH)$ be the set of independent sets of $\HH$. 

\begin{lemma}[Container Lemma \cite{BMS, ST}]\label{container}
For every $a \in \NN$ and $c > 0$ there exists a $\delta > 0$ such that the following holds.  Let $\tau \in (0, 1)$ and suppose $\HH$ is a $a$-uniform hypergraph on $N$ vertices such that $$\Delta_{b}(\HH) \leq c \tau^{b - 1}\frac{|\HH|}{N}.$$

for every $1 \leq b \leq a$.   Then there exists a family $\C$ of subsets of $V(\HH)$ and a function $f: 2^{V(\HH)} \rightarrow \C$ such that: 
\begin{enumerate}
\item For every $I \in \I(\HH)$, there is a $T(I) \subseteq I$ with $|T(I)| \leq a \cdot \tau N$ and $I \subseteq f(T(I)) \cup T(I)$. 
\item $|C| \leq (1 - \delta) N$ for every $C \in \C$. 
\end{enumerate}
\end{lemma}

Given a poset $P$, let $\G_P$ be the $|P|$-uniform hypergraph with vertex set $\B_n$, where a set is hyperedge if the poset the set induces is isomorphic to $P$. 

Fix a tree poset $P$ of height $k$. For subsets $\F\sub \B_n$ we define $$\tau(\F, k) : = \begin{cases}
    \frac{1}{n} & \text{if } |\F| < 3k \binom{n}{n / 2}\\
    \frac{1}{n^3} & \text{if } |\F| \geq 3k \binom{n}{n / 2}.
\end{cases}$$

We will use Theorem~\ref{thm:balanced} to prove the following result: 
\begin{corollary}\label{cor:impt}
Let $P$ be a tree poset of height $k$. Then for every $\ve > 0$, there exists $\delta = \delta(\ve, k) > 0$ such that the following holds.  Let $n \in \NN$ be sufficiently large and $\F \subseteq \B_n$ with $|\F| \geq (k - 1 + \ve) \binom{n}{ n / 2}$.  Then there exists a collection $\C \subseteq 2^{\F}$ and a function $f: 2^{\F} \rightarrow \C$  such that: 
\begin{enumerate}
\item For every set $I \in \I(\G_P[\F])$, there exists a $T$ with $|T(I)| \leq |P| \cdot \tau(\F, k)|\F|$ and $T(I) \subseteq I  \subseteq f(T(I)) \cup T(I)$. 
\item For every $C \in \C$,  $|C| \leq (1 - \delta ) |\F|$.  
\end{enumerate}
\end{corollary}
\begin{proof}
    If $(k - 1 + \ve) \binom{n}{n / 2} \leq |\F| \leq 3k \binom{n}{n /2}$, then we apply Theorem~\ref{thm:balanced} with $\ve$, $\ell = 1$, to find a $\delta_1 := \delta_1(\ve, P)$ and a subgraph $\HH \subseteq \G_P[\F]$ with properties \begin{enumerate}
        \item $|\HH| \geq \delta_1^{|P|} n^{|P| - 1}\binom{n}{ n / 2}$,
        \item $\Delta_j(\HH) \leq (\delta_1 n)^{|P| - j}$ for all $1 \leq j \leq |P| $.
    \end{enumerate} Note this satisfies the conditions of Lemma~\ref{container}, with $c = \frac{2}{\delta_1^{|P|}}$ and $\tau = \tau(\F, k) = \frac{1}{n}$. Applying Lemma~\ref{container} to $\HH$ then gives the result.

    If $|\F| \geq 3k \binom{n}{n / 2}$, then we arbitrarily split $\F$ into $t + 1$ families $\F = \F_0 \cup \F_1 \cup \dots \cup \F_t$ such that $|\F_0| < 3k \binom{n}{n / 2}$ and each $|\F_i| = 3k \binom{n}{ n / 2}$. For each $\F_i$ and $j \in \{0, 1, 2\}$, define $$\F_i^j := \{F \in \F_i : |F| \equiv j \mod 3\}.$$ Since the $\F_i^j$ partition $\F_i$, there exists some $j$ where $|\F_i^j| \geq k \binom{n}{n / 2}$. For this $j$, let  $\F_i' = \F_i^j$.

    Since $\F_i'$ is a $3$-gapped family, we may apply Theorem~\ref{thm:balanced} to $\F_i'$ with $\ve = \frac{1}{2}$, $\ell = 3$, to find a $\delta_2 : = \delta_2(\ve, P)$ a subgraph $\HH_i \subseteq \G_P[\F_i]$ with properties 
    \begin{enumerate}
        \item $|\HH| \geq \delta_2^{|P|} n^{3(|P| - 1)}\binom{n}{ n / 2}$, 
        \item $\Delta_j(\HH) \leq (\delta_2 n)^{3(|P| - j)}$ for all $1 \leq j \leq |P| $. 
    \end{enumerate} Consider now the collection $\HH = \cup_{i = 1}^t \HH_i \subseteq \G_P[\F]$, noting that because the $\HH_i$ are vertex disjoint we have $\Delta_\ell(\HH)=\max_i \Delta_\ell(\HH_i)$ for all $1 \leq b \leq |P|$, and hence \begin{align*}\Delta_b(\HH)= \max_j \Delta_{b}(\HH_j)= (\delta_2 n)^{3(|P|-b)} \leq \frac{6k}{\delta_2^{|P| }} n^{3(1-b)}\cdot \frac{t \delta_2^{|P| } n^{3(|P|-1)}{n\choose n/2}}{6k t{n\choose n/2}} \leq \frac{6k}{\delta_2^{|P|}} n^{3(1-b)}\cdot \frac{|\HH|}{|\F|},\end{align*}
    where this last step used that $|\F| \leq 3 k ( t+ 1) {n\choose n/2} \leq 6k t \binom{n}{n / 2}$ by how we defined $t$ and that $|\HH|=\sum |\HH_j| \geq t\cdot\delta_2^{|P|} n^{3(|P|-1)}{n\choose n/2}$.  With this we can apply Lemma~\ref{container} to $\HH$ with $c = \frac{6k}{\delta_2^{|P|}}$ and $\tau = \tau(\F, k) = \frac{1}{n^3}$.
\end{proof}

In this section, we follow directly the framework set up in \cite{CM} following \cite{BMS, ST}.

\begin{definition}
A fingerprint of $\G_P$ is a triple $(\mathcal{T},g,C)$ such that:
\begin{enumerate}
\item $\mathcal{T}$ is a collection of ``certificates'', which are vectors $T=(T_1,\ldots,T_m)$ of disjoint subsets of $V(\G_P)$.  For such a vector we let $\widehat{T}=\bigcup_i T_i$.

\item $g: \I(\G_P) \rightarrow \mathcal{T}$ is a ``fingerprint function'' which satisfies $\widehat{g(I)} \subseteq I$ for every $I \in \I(\G_P)$. 
\item $C: \mathcal{T} \rightarrow 2^{V(\G_P)}$ is a ``container function'' such that $I \subseteq C(g(I))$ for every $I \in \I(\G_P)$.  
\end{enumerate}
\end{definition}

We will now apply \Cref{cor:impt} iteratively to construct a fingerprint of ``small" size. This fingerprint will then be directly used to prove \Cref{thm:counting} and \Cref{thm:randomturan}. 

\begin{theorem}\label{thm:major}
For every tree poset $P$ of height $k \geq 2$ and every $\ve > 0$, there exists a constant $K = K(\ve, P) > 0$ and a fingerprint $(\mathcal{T}, g, C)$ of $\G_P$ such that if $n$ is sufficiently large, the following holds: 
\begin{enumerate}
\item[(a)] Every $T \in \mathcal{T}$ satisfies $|\widehat{T}| \leq \frac{K}{n} \binom{n}{n /2}$.
\item[(b)] The number of members $T\in \mathcal{T}$ with $|\widehat{T}|=s$ is at most 
$$\left( \frac{K \binom{n}{n /2}}{s} \right)^s \exp\left( \frac{K}{n} \binom{n}{n / 2}\right).$$
\item[(c)] $|C(g(I))| \leq (k - 1 + \ve) \binom{n}{ n / 2}$ for every $I \in \I(\G_P)$.  
\end{enumerate}
\end{theorem}

To prove part (b) of this theorem, we will use the Lemma 4.3 from \cite{CM}, stated below.

\begin{lemma}\label{lem:xlogx}
Let $M > 0, s > 0$ and $0 < \delta  <1$.  For any sequence $(a_1, a_2,  \dots, a_m)$ of real numbers summing up to $s$ such that $1 \leq a_j \leq (1 - \delta)^j M$ for each $j \in [m]$, we have 
$$s \log s \leq \sum_{j = 1}^m a_j \log a_j + O(M)$$.  
\end{lemma}
\begin{proof}[Proof of \Cref{thm:major}]
Let $\delta := \delta(\ve, P)$ be given by Corollary~\ref{cor:impt}, and choose $K$ large depending on $\ve, \delta, |P|, k$.  Let $n$ be sufficiently large. 

Fix some $I \in \I(\G_P)$, we will apply Corollary~\ref{cor:impt} a certain number of times, which we will denote by $m = m(I)$, to construct two sequences $\F_1, \F_2, \dots \F_{m + 1}$ and $T_1,  T_2,  \dots T_m$ of subsets of $V(\G_P)$. 

First set $\F_1 := \B_n$. Then apply \Cref{cor:impt} to $\F_1$ with $\frac{\ve}{2}$. This gives us a $T(I)$ of size less than $|P| \tau(\F_1, k) |\F_1|$ and a $f(T(I))$ of size less than $(1 - \delta)|\F_1|$ such that $I \subseteq T(I) \cup f(T(I))$. Let $T_1 = T(I)$ and $\F_2 = f(T(I)) - T(I)$. In general, as long as $|\F_i| \geq (k - 1 + \frac{\ve}{2})\binom{n}{ n / 2}$, apply \Cref{cor:impt} to $\F_i$ with $\frac{\ve}{2}$. This gives a $T(I)$ and a $f(T(I))$ such that $T(I) \subseteq \F_i \cap I$ and $\F_i \cap I \subseteq f(T(I)) \cup T(I)$. Let $T_i = T(I)$, and $\F_{i + 1} = f(T(I)) - T(I)$.

It is easy to see that above construction will inductively maintain the following properties for all $i$: 
\begin{enumerate}
\item[(i)] $I \subseteq \F_{i + 1} \cup T_1 \cup T_2 \cup \dots \cup T_i$,
\item[(ii)] $ \F_{i + 1}, T_1, \dots T_i$ are pariwise disjoint,
\item[(iii)] $\F_{i + 1}$ depends only on $\F_i$ and $T_i$, 
\item[(iv)] $|\F_{i + 1}| \leq (1 - \delta) |\F_i|$.
\end{enumerate}

We define our fingerprint $(\mathcal{T},  g,  C)$ of $\G_P$ by setting.

$$g(I) := (T_1,  T_2,  \dots T_m) \text{ and } C(g(I)) := \F_{m + 1} \cup T_1 \cup \dots \cup T_m, $$
and letting $\mathcal{T} := \{ g(I) : I \in \I(\G_P)\}$.  Note that property (iii) implies that $C$ is well defined, as the choice of $\F_{m + 1}$ does not depend on $I$, while property (i) guarantees that it is a container function.  Similarly (ii) together with how we constructed $T_i$ guarantees that $g$ is a fingerprint function.

In order to check that the constructed fingerprint satisfies the conditions of the theorem,  we first bound the sizes of the $T_i$'s, and then the number of iterations.

To begin,  let $2 \leq m_0 \leq m$ be minimum such that $|\F_{m_0}| \leq 3 k \binom{n}{n / 2}$ and observe that by property (iv) and definition of $\tau(\F, k)$: 

$$\tau(\F_i, k) |\F_i| \leq \begin{cases} n^{-3} \cdot 2^n & \text{if } i < m_0, \\
n^{-1} \cdot (1 - \delta)^{ i - m_0} 3k \binom{n}{n / 2} & \text{otherwise} . \end{cases}$$

Since $|\F_i|$ decays at a geometric rate by (iv),  we have that $m = O_{\ve, P}(\log n )$ and $m - m_0 = O_{\ve, P}(1)$.

We thus have the following 

\begin{equation}
\sum_{i = 1}^{m_0 - 1} \tau (\F_i, k) |\F_i| \leq \frac{m 2^n}{n^3} \ll \frac{1}{n^2} \binom{n}{n / 2} \text{ and } \sum_{i = m_0}^{m} \tau (\F_i, k) |\F_i| = \frac{O_{\ve, P}(1)}{n} \binom{n}{ n / 2}
\end{equation}

Thus by \Cref{cor:impt}, $|T_i| \leq |P| \tau (\F_i, k) |\F_i|$, and therefore
\begin{equation}\label{TauFBound}
\sum_{i = 1}^{m_0 - 1} |T_i| \leq \frac{m 2^n}{n^3} \ll \frac{1}{n^2} \binom{n}{n / 2} \text{ and } \sum_{i = m_0}^{m} |T_i| = \frac{O_{\ve, P}(1)}{n} \binom{n}{ n / 2}
\end{equation}

Since $|\widehat{g(I)}| =  \sum_{i = 1}^m |T_i|$,  we have the bound in (a). 

Similarly, $n$ is sufficiently large 
$$|C(g(I)) = |\F_{m + 1}| + |T_1 \cup \dots \cup T_m| \leq (k - 1 +  \ve) \binom{n}{n / 2},$$ proving part (c) of the theorem.  

It remains to prove (b). Recall that $\T$ is a collection of sequences $T = (T_1, T_2, \dots T_m)$ where $m$ can be arbitrary. We are looking to count the number $T \in \T$ such that $|\widehat{T}| = \sum_{i = 1}^{m} |T_i| = s$. To do this, we will partition $\mathcal{T}$ into subfamilies of form $\T(m_0, \mathbf{t})$ for all $m_0 \in \NN$ and $\mathbf{t} = (t_1, t_2, \dots, t_{m})$ in $\NN^m$. We collect in $\T(m_0, \mathbf{t})$ all $T = (T_1, T_2, \dots T_{m}) \in \T$ such that $|T_i| = t_i$ for all $i \in [m]$ and  $m_0$ is the minimum integer for which $|\F_{m_0}| \leq 3k \binom{n}{n /2}$. Notice that while $\F_{m_0}$ were produced for various $I$, a fixed $\F_{m_0}$ depends only on $T_1, T_2, \dots T_{m_0}$ by property (iii).

Let $s_1 = \sum_{i = m_0}^{m} t_i$,  and observe that Lemma~\ref{lem:xlogx} applied with $M = \frac{3k}{n} \binom{n}{ n / 2}$, $s = s_1$, and $\delta$ implies  

\begin{align*}
    &\sum_{i = m_0}^{m} t_i \log t_i \geq s_1 \log s_1 + \frac{O_{\ve, P}(1)}{n} \binom{n}{n  /2}  \implies \\
    &\sum_{i = m_0}^{m} t_i \log \frac{1}{t_i} \leq s_1 \log \frac{1}{s_1} + \frac{O_{\ve, P}(1)}{n} \binom{n}{n  /2}
\end{align*}
We also observe that by \eqref{TauFBound} we have that
\begin{equation}\sum_{i=1}^{m_0-1} t_i \ll \frac{1}{n^2} {n\choose n/2}.\label{eq:TFBound}\end{equation}
Using all this together with the observation that each $T_i$ is a subset of the corresponding $\F_i$ for all $i$, together with $|\F_i| \leq 2^n$ and the definition of $m_0$, we find that 
\begin{align*}
|\mathcal{T}(m_0,  \mathbf{t})| & \leq \prod_{i = 1}^{m_0 - 1} \binom{2^n}{t_i} \prod_{i = m_0}^{m} \binom{3k \binom{n }{n / 2}}{t_i}\\
&\leq \left(  \prod_{i = 1}^{m_0 - 1} 2^{t_i n} \right) \left( \left[ 3e k \binom{n}{n /2} \right]^{s_1} \prod_{i = m_0}^{m} \left( \frac{1}{t_i}\right)^{t_i}\right)\\
&= \exp\left( \log(2)n \sum_{i = 1}^{m_0 - 1} t_i+s_1\log\left( 3e k \binom{n}{n /2} \right)+\sum_{i = m_0}^{m} t_i \log \frac{1}{t_i}\right)\\
&\leq\exp\left(\frac{O_{\ve, P}(1)}{n} \binom{n}{n  /2} +s_1 \log\left(3e k \binom{n}{n /2} \right) +  s_1 \log \frac{1}{s_1} + \frac{O_{\ve, P}(1)}{n} \binom{n}{n  /2}  \right)\\
&=  \left( \frac{3ek }{s_1} \binom{n}{n / 2} \right)^{s_1} \exp \left(  \frac{O_{\ve, P}(1)}{n} \binom{n}{n  /2}  \right).
\end{align*}

Since this final expression is a monotone function in $s_1$ over the interval $(0,  \frac{1}{3ek} \binom{n}{n / 2})$,  we may replace $s_1$ with the larger value $s\le \frac{K}{n}{n\choose n/2}$ (with this bound for $s$ using (a)).  There are only $2^{O(n \log n)}$ choices for $\mathbf{t}$  and $O(\log (n))$ choices for both $m_0$ and $m$, so the claimed bound holds by summing the bound above over all possible $\mathcal{T}(m_0,\mathbf{t})$.  
\end{proof}

\subsection{Applications of Container Lemma}

We will now prove \Cref{thm:counting}. 

\begin{theorem}[Restatement of \Cref{thm:counting}]
    For every $\ve > 0$ and $P$ a tree poset of height $k$, there is an there exists an $n_0$ such that if $n \geq n_0$, then the number of induced $P$-free sets in $\B_n$ is at most 
    $$2^{(k - 1 + \ve)\binom{n}{n / 2}}.$$
\end{theorem}
\begin{proof}
    Recall that $\mathcal{G}_P$ is the hypergraph on vertex set $\B_n$ where each edge corresponds to an induced copy of $P$, and hence the number of induced $P$-free subsets of $\B_n$ is exactly the number of independent sets of $\G_P$.
    
    Now, apply Theorem~\ref{thm:major} to $\mathcal{G}_P$ with $\ve / 2, P$ to find a constant $K = K(\ve, P) > 0$ and a fingerprint $(\mathcal{T}, g, C)$ of $\G_P$ such that the following holds if $n$ is sufficiently large: 
\begin{enumerate}
\item Every $T \in \mathcal{T}$ satisfies $|\widehat{T}| \leq \frac{K}{n} \binom{n}{n /2}$
\item The number of members of $\mathcal{T}$ of size $s$ is at most 
$$\left( \frac{K \binom{n}{n /2}}{s} \right)^s \exp\left( \frac{K}{n} \binom{n}{n / 2}\right)$$
\item $|C(g(I))| \leq (k - 1 + \frac{\ve}{2}) \binom{n}{ n / 2}$ for every $I \in \I(\G_P)$ 
\item Every $I \in \mathcal{I}(\G_P)$ satisfies $I \subseteq C(g(I))$. 
\end{enumerate}

Thus, the number of independent sets $I \in \mathcal{I}(\G_P)$ at most the sum over all $T \in \mathcal{T}$ the number of subsets of $C(T)$, and so we have the following bound:

\begin{align*}
    |\I(\G_P)| & \leq \sum_{s = 1}^{\frac{K}{n} \binom{n}{n / 2}} |\{T \in \mathcal{T} : |T| = s\} | 2^{(k - 1 + \frac{\ve}{2}) \binom{n}{ n / 2}}\\
    & \leq  \sum_{s = 1}^{\frac{K}{n} \binom{n}{n / 2}} \left( \frac{K \binom{n}{n /2}}{s} \right)^s \exp\left( \frac{K}{n} \binom{n}{n / 2}\right) 2^{(k - 1 + \frac{\ve}{2}) \binom{n}{ n / 2}}\\
    &\leq \frac{K}{n} \binom{n}{n / 2} n^{\frac{K}{n} \binom{n}{n / 2}}\exp\left( \frac{K}{n} \binom{n}{n / 2}\right)2^{(k - 1 + \frac{\ve}{2}) \binom{n}{ n / 2}}\\
    &\leq 2^{(k - 1 + \ve) \binom{n}{ n / 2}}, 
\end{align*}
as desired.
\end{proof}

\begin{theorem}[Restatement of \Cref{thm:randomturan}]
    Let $P$ be a tree poset of height $k$. Let $\Pa(n,  p)$ be the uniformly random subset of $\B_n$,  where each set survives with probability $p$ such that $pn \rightarrow \infty$. Then with high probability, the largest induced $P$-free subset of $\Pa(n, p)$ has size $(k - 1 + o(1))p\binom{n}{n / 2}$. 
\end{theorem}

\begin{proof}
Let $k \geq 2$ and $\ve > 0$ be arbitrary,  and let $K = K(\ve,  P) > 0$ and $(\mathcal{T},  g,  C)$ be the constant and fingerprint given by Theorem~\ref{thm:major} applied to $\mathcal{G}_P$ with $\ve, P$.  Let $n \in \NN$ be sufficiently large,  and note that we may assume $pn \geq K \ve^{ - 1}$ since $pn \rightarrow \infty$.  Let $\Pa(n,  p)$ be the uniformly random subset of $\B_n$,  where each set survives with probability $p$ .  Suppose $I \subseteq \Pa(n, p)$ is an induced $P$-free subset (or equivalently an independent set of $\G_P$) of size at least $(k - 1 + 3 \ve) p \binom{n}{ n / 2}$.   Then, it follows that $\widehat{g(I)} \subseteq \Pa(n, p)$ and 

$$|C(g(I)) \cap \Pa(n, p)| \geq (k - 1 + 3 \ve) p \binom{n}{n/2}.$$

Let $X$ be the number of elements of $\mathcal{T}$ for which these two properties  $\widehat{T}\sub \Pa(n,p)$ and $|C(T)\cap \Pa(n,p)|\ge (k-1+3\ve)p{n\choose n/2}$ hold.  Then, 

$$\EE(X) \leq \sum_{T \in \mathcal{T}} \prob(\widehat{T} \subseteq \Pa(n, p) )\cdot  \prob\left(|C(T)  \cap \Pa(n, p)-\widehat{T}| \geq (k - 1 + 2 \ve) p \binom{n}{n / 2}\right)$$

where we used that $|\widehat{T}| \leq 
\frac{K}{n} \binom{n}{ n / 2} \leq \ve p \binom{n}{n / 2}$ and $\widehat{T} \subseteq C(T)$ by the lower bound on $pn$ and Theorem~\ref{thm:major}.

Note, by \Cref{lem:chernoff}, applied with $\delta = \frac{\ve}{k - 1 + \ve}$ we obtain the following: 

\begin{align*}
    \prob\left(|C(T)  \cap \Pa(n, p)-\widehat{T}| \geq (k - 1 + 2 \ve) p \binom{n}{n / 2}\right) \leq 2 \exp\left( - \frac{\ve^2}{3k} p  \binom{n}{ n / 2}\right)
\end{align*}

Hence,  by properties of $(\mathcal{T}, g, C)$ guaranteed by Theorem~\ref{thm:major} and the above inequality, we have:

\begin{align*}
\EE(X) & \leq \sum_{s = 1}^{\frac{K}{n} \binom{n}{ n / 2}} \left( \frac{K \binom{n}{n /2}}{s} \right)^s \exp\left( \frac{K}{n} \binom{n}{n / 2}\right) \cdot p^s \cdot 2 \exp\left( - \frac{\ve^2p}{3 k} \binom{n}{ n / 2} \right) \\
& \leq \sum_{s = 1}^{\frac{K }{n} \binom{n}{ n / 2}} \left( \frac{K p \binom{n}{n /2}}{s} \right)^s \exp\left( \frac{K}{n} \binom{n}{n / 2}\right)  \cdot 2 \exp\left( - \frac{\ve^2p}{3 k} \binom{n}{ n / 2} \right) \\
& \leq \frac{2 K}{n} \binom{n}{ n / 2} \exp\left( \frac{K}{n} \binom{n}{n / 2} (\log (pn) + 1) - \frac{\ve^2p}{3 k} \binom{n}{ n / 2}  \right) 
\end{align*}

  Therefore by Markov's inequality and $pn \gg \log (pn) \gg 1$ we have that $$\prob\left( \alpha(\G_P[\Pa(n, p)]) \geq (k - 1 + 3 \ve) p \binom{n}{n /2} \right) \leq \exp \left( -  \frac{\ve^2 p }{6 k} \binom{n}{n / 2} \right)\rightarrow 0 $$ as $n \rightarrow \infty$, as required.

\end{proof}

\end{document}